\documentclass[12pt, oneside]{amsart}
\usepackage[utf8]{inputenc}
\usepackage[english]{babel}
\usepackage{amssymb,amsmath,amsthm}
\usepackage{graphicx}
\usepackage{xspace}
\usepackage{enumerate}
\usepackage{verbatim}
\usepackage[all, cmtip,arrow]{xy}
\usepackage{geometry} % Allows to change page margins
\newcommand{\F}[1]{\ensuremath{\mathbb{F}_{#1}}}

\newcommand{\QQ}{\ensuremath{\mathbb{Q}}\xspace}

\newcommand{\ZZ}{\ensuremath{\mathbb{Z}}\xspace}
\newcommand{\Zp}{\ensuremath{\mathbb{Z}_{(p)}}\xspace}
\newcommand{\Z}[1]{\ensuremath{\mathbb{Z}_{(#1)}}\xspace}
\newcommand{\LL}{\ensuremath{\mathbb{L}}\xspace}

\newcommand{\OO}{\ensuremath{\mathcal{O}}\xspace}
\newcommand{\rarr}{\rightarrow}

\newcommand{\xrarr}[1]{\xrightarrow{#1}}

\newcommand{\ot}{\otimes}

\newcommand{\Tor}{\ensuremath{\mathrm{Tor}}}

\newcommand{\gr}{\mathrm{gr}}

\numberwithin{equation}{section}

\newtheorem{thm}[equation]{Theorem}
\newtheorem{prop}[equation]{Proposition}
\newtheorem{lem}[equation]{Lemma}
\newtheorem{cor}[equation]{Corollary}
\newtheorem*{conj}{Guiding principle}

\theoremstyle{definition}
\newtheorem{rem}[equation]{Remark}

\newtheorem{dfn}[equation]{Definition}
\newtheorem{ntt}[equation]{}

\newcommand{\llarrow}{\mathrel{\vcenter{\vbox{\offinterlineskip
\hbox{$\leftarrow$}\hbox{$\leftarrow$}}}}}
\newcommand{\lllarrow}{\mathrel{\vcenter{\vbox{\offinterlineskip
\hbox{$\leftarrow$}\hbox{$\leftarrow$}\hbox{$\leftarrow$}}}}}

\newcommand{\zz}{\mathbb{Z}}
\newcommand{\aff}{\mathrm{aff}}
\newcommand{\qq}{\mathbb{Q}} 
\newcommand{\laz}{\mathbb{L}}
\newcommand{\ff}{\mathbb{F}}

\newcommand{\ta}{\mathbb{T}}
\newcommand{\res}{\mathrm{res}}
\DeclareMathOperator{\codim}{\mathrm{codim}}

\DeclareMathOperator{\End}{\mathrm{End}}

\newcommand{\A}{\mathrm{A}}
\newcommand{\B}{\mathrm{B}}
\newcommand{\C}{\mathrm{C}}
\newcommand{\D}{\mathrm{D}}
\newcommand{\E}{\mathrm{E}}
\newcommand{\FF}{\mathrm{F}}
\newcommand{\G}{\mathrm{G}}

\DeclareMathOperator{\Aut}{\mathrm{Aut}}

\DeclareMathOperator{\SL}{\mathrm{SL}}
\DeclareMathOperator{\Sp}{\mathrm{Sp}}
\DeclareMathOperator{\Spec}{\mathrm{Spec}}
\DeclareMathOperator{\CH}{\mathrm{CH}}
\DeclareMathOperator{\Ch}{\mathrm{Ch}}
\DeclareMathOperator{\Pic}{\mathrm{Pic}} 
\DeclareMathOperator{\ind}{\mathrm{ind}} 

\newcommand{\MS}{\mathrm{MS}}
\newcommand{\Nrd}{\mathrm{Nrd}}
\newcommand{\SB}{\mathrm{SB}}
\newcommand{\Br}{\mathop{\mathrm{Br}}}
\newcommand{\Spin}{\operatorname{\mathrm{Spin}}}
\newcommand{\GL}{\operatorname{\mathrm{GL}}}

\newcommand{\Fsep}{F_{{\mathrm{sep}}}}
\newcommand{\sep}{\mathrm{sep}}

\newcommand{\Ker}{\operatorname{Ker}}
\newcommand{\et}{{\text{et}}}

\title[Applications of the Morava $K$-theory to algebraic groups]{Applications of the Morava $K$-theory to algebraic groups}
\author{Pavel Sechin and Nikita Semenov}
\thanks{The authors gratefully acknowledge the support of SPP 1786 ``Homotopy theory and algebraic geometry'' (DFG).
The first author was partially supported by the Russian Academic Excellence Project '5-100',
by the Moebius Contest Foundation for Young Scientists, by Simons Foundation and by DFG-FOR 1920.
The second author gratefully acknowledges Universit\'e Paris~13.}


\date{}

\begin{document}

\maketitle

\begin{comment}
\begin{abstract}
In the present article we discuss an approach to cohomological invariants of algebraic groups over fields of characteristic zero
based on the Morava $K$-theories, which are generalized oriented cohomology theories in the sense
of Levine--Morel.

We show that the second Morava $K$-theory detects the triviality of the Rost invariant and, more generally, relate the triviality of cohomological invariants and the splitting of Morava motives.

We describe the Morava $K$-theory of generalized Rost motives, compute the Morava $K$-theory of
some affine varieties, and characterize the powers of the fundamental ideal of the Witt ring with the help of the
Morava $K$-theory. Besides, we obtain new estimates on torsion in Chow groups of codimensions up to $2^n$ of quadrics from the $(n+2)$-nd power of the fundamental ideal of the Witt ring. We compute torsion in Chow groups of $K(n)$-split varieties with respect to a prime $p$ in all codimensions up to $\frac{p^n-1}{p-1}$ and provide a combinatorial tool to estimate torsion up to codimension $p^n$. An important role in the proof is played by the gamma filtration on Morava $K$-theories, which gives a conceptual explanation of the nature of the torsion.

Furthermore, we show that under some conditions the $K(n)$-motive of a smooth projective variety splits if and only if its $K(m)$-motive splits for all $m\le n$.
\end{abstract}
\end{comment}

\begin{abstract}
In the article we discuss an approach to cohomological invariants of algebraic groups
based on the Morava $K$-theories.

We show that the second Morava $K$-theory detects the triviality of the Rost invariant and, more generally, relate the triviality of cohomological invariants and the splitting of Morava motives.

We compute the Morava $K$-theory of generalized Rost motives and of
some affine varieties and characterize the powers of the fundamental ideal of the Witt ring with the help of the
Morava $K$-theory. Besides, we obtain new estimates on torsion in Chow groups of quadrics and investigate torsion in Chow groups of $K(n)$-split varieties. An important role in the proofs is played by the gamma filtration on Morava $K$-theories, which gives a conceptual explanation of the nature of the torsion.

Furthermore, we show that under some conditions if the $K(n)$-motive of a smooth projective variety splits, then its $K(m)$-motive splits for all $m\le n$.
\end{abstract}

\tableofcontents

\newpage

\section{Introduction}
The present article is devoted to applications of the Morava $K$-theory to cohomological invariants of algebraic groups
 and to computations of the Chow groups of quadrics.

\begin{ntt}[Cohomological invariants]

In his celebrated article on irreducible representations \cite{Ti71} 
Jacques Tits introduced the notion of a Tits algebra, which is
an example of cohomological invariants of algebraic groups of degree $2$. 
This invariant of a linear algebraic
group $G$ plays a crucial role in the computation of the $K$-theory of twisted 
flag varieties by Panin \cite{Pa94} and in the index reduction formulas by Merkurjev,
Panin and Wadsworth \cite{MPW96}. It has important applications to the classification of linear algebraic groups
and to the study of associated homogeneous varieties.

The idea to use cohomological invariants in the classification of algebraic groups goes back to Jean-Pierre Serre.
In particular, Serre conjectured the existence of an invariant of degree $3$ 
for groups of type $\FF_4$ and $\E_8$.
This invariant was later constructed by Markus Rost for all $G$-torsors, where
$G$ is a simple simply-connected algebraic group, and is now called the Rost invariant (see \cite{GMS03}).

Moreover, the Serre--Rost conjecture for groups of type $\FF_4$ says that the map
$$H_{\et}^1(F,\FF_4)\hookrightarrow H_{\et}^3(F,\zz/2)\oplus H_{\et}^3(F,\zz/3)\oplus H_{\et}^5(F,\zz/2)$$
induced by the invariants $f_3$, $g_3$ and $f_5$ described in \cite[\S40]{Inv} ($f_3$ and $g_3$ are the modulo $2$ and modulo $3$
components of the Rost invariant), is injective. 
The validity of the Serre--Rost conjecture would imply that
one can exchange the study of the {\sl set} $H_{\et}^1(F,\FF_4)$ of isomorphism classes of groups of type $\FF_4$ over $F$ 
(equivalently of isomorphism classes of $\FF_4$-torsors or of isomorphism classes of Albert algebras)
by the study of the {\sl abelian group} $H_{\et}^3(F,\zz/2)\oplus H_{\et}^3(F,\zz/3)\oplus H_{\et}^5(F,\zz/2)$.

In the same spirit one can formulate the Serre conjecture II, 
saying in particular that $H_{\et}^1(F,\E_8)=1$ if
the field $F$ has cohomological dimension $2$. Namely, for such fields $H_{\et}^n(F,M)=0$ for all $n\ge 3$ and all torsion modules
$M$. In particular, for groups over $F$ there are no invariants of degree $\ge 3$, and the Serre conjecture II predicts
that the groups of type $\E_8$ over $F$ themselves are split.

Furthermore, the Milnor conjecture on quadratic forms (proven by Orlov, Vishik and Voevodsky) together with the Milnor conjecture on the \'etale cohomology (proven by Voevodsky)
provides a classification of quadratic forms over fields
in terms of the Galois cohomology, i.e., in terms of cohomological invariants. 

In the present article we will relate the Morava $K$-theory with some cohomological invariants of algebraic groups.
\end{ntt}

\begin{ntt}[Morava $K$-theory and Morava motives]
Let $n$ be a positive integer and let $p$ be a prime.
The Morava $K$-theory $K(n)^*$ is a free oriented cohomology theory
in the sense of Levine--Morel \cite{LM} whose coefficient ring is $\zz_{(p)}$,
whose formal group law
has height $n$ modulo $p$, and
the logarithm is of the type
$$\log_{K(n)}(x)=x+\frac{a_1}{p}x^{p^n}+\frac{a_2}{p^2}x^{p^{2n}}+\ldots$$ with $a_i\in\zz_{(p)}^\times$. 
If $n=1$ and all $a_i$ are equal to $1$, then 
the theory $K(1)^*$ is isomorphic to Grothendieck's $K^0\otimes\zz_{(p)}$ as a presheaf of rings. 
Moreover, there is some kind of analogy between
Morava $K$-theory in general and $K^0$.

More conceptually, algebraic cobordism of Levine--Morel can be considered as a functor to
 the category of comodules over the Hopf algebroid $(\LL,\LL B)$, 
where $\LL$ is the Lazard ring and $\LL B=\LL[b_1,b_2,\ldots]$.
This Hopf algebroid parametrizes the groupoid of formal group laws
with strict isomorphisms between them, 
and the category of comodules over it
can be identified with the category of quasi-coherent sheaves 
over the stack of formal groups $\mathcal{M}_{fg}$. This stack modulo $p$
has a descending filtration by closed substacks $\mathcal{M}_{fg}^{\ge n}$ which classify the formal group laws of height $\ge n$.
Moreover, $\mathcal{M}_{fg}^{\ge n}\setminus\mathcal{M}_{fg}^{\ge n+1}$ has an essentially unique geometric point which corresponds to
the Morava $K$-theory $K(n)^*\otimes\overline{\mathbb{F}}_{p}$. This chromatic picture puts $K(n)^*$ into an intermediate position between $K^0$ and $\CH^*$.

We remark also that Levine and Tripathi construct in \cite{LT15} a higher Morava $K$-theory in algebraic geometry.
\end{ntt}
                       
\begin{ntt}[Morava $K$-theories, split motives and vanishing of cohomological invariants]

There are three different types of results in this article
which fit 
into the following guiding principle.
The leading idea of this principle has been probably well understood already by Voevodsky, since he considered the Morava $K$-theory in his program on the proof of the Bloch--Kato conjecture in \cite{Vo95}.

\begin{conj}
Let $X$ be a projective homogeneous variety, let $p$ be a prime number
and let $K(n)^*$ denote the corresponding Morava $K$-theory.

Then vanishing of cohomological invariants of $X$ with $p$-torsion coefficients
in degrees no greater than $n+1$ {\sl should} correspond to 
 the splitting of the $K(n)^*$-motive of $X$.
\end{conj}

First of all, due to 
the Milnor conjecture 
the associated graded ring of the Witt ring $W(F)$ of a field $F$ of characteristic not $2$
is canonically isomorphic to the \'etale cohomology of the base field 
with $\ZZ/2$-coefficients: $H^n_{\et}(F,\ZZ/2)\simeq I^n/I^{n+1}$, where
$I$ denotes the fundamental ideal of $W(F)$.
Therefore, the projective quadric which corresponds to a quadratic form $q\in I^n$
has a canonical cohomological invariant of degree $n$. 
The guiding principle suggests that the $K(n)^*$-motive of an even-dimensional projective quadric is split
if and only if 
the class of the corresponding quadratic form in the Witt ring lies in the ideal $I^{n+2}$.
Indeed, we prove this statement in Proposition~\ref{moravaproj}.

Secondly, we relate cohomological invariants of simple algebraic groups to Morava $K$-theories.
We show in Section~\ref{sec:morava_coh_inv} that for a simple simply-connected group $G$ with trivial Tits algebras the Morava $K$-theory $K(2)^*$ 
detects the triviality of the Rost invariant of $G$. 
Note that in a similar spirit Panin showed in \cite{Pa94} 
that the Grothendieck's $K^0$ detects the triviality of Tits algebras.
Moreover, for a group $G$ of type $\E_8$ the Morava $K$-theory $K(4)^*$ for $p=2$ 
detects the splitting of the variety of Borel subgroups of $G$ over
a field extension of odd degree (Theorem~\ref{moravarost}).
All these results agree with the guiding principle.

Thirdly, we relate the property of being split with respect to Morava $K$-theories $K(n)^*$ 
for different $n$.
Namely, we prove in Proposition~\ref{prop:height_tate_motive} 
that if a smooth projective geometrically cellular variety $X$ over a field $F$
of characteristic $0$ satisfies the Rost nilpotence principle for Morava $K$-theories 
and has a split $K(n)^*$-motive,
then it has a split $K(m)^*$-motive for all $m\le n$.
In particular,
Morava motives provide a linearly ordered
series of obstructions for a projective homogeneous variety
to be isotropic over a base field extension of a prime-to-$p$ degree.
\end{ntt}                       
                                                                                          
\begin{ntt}[Operations in Morava $K$-theories and applications to Chow groups]

The study of cohomological invariants of algebraic groups is partially motivated
by the interest in Chow groups of torsors. Whenever some cohomological invariants vanish,
one may ask whether this yields any restrictions on the structure of Chow groups,
e.g. existence, order or cardinality of torsion in certain codimensions.
We approach this question by studying projective 
homogeneous  (or, more generally, geometrically cellular) varieties $X$ for which the $K(n)^*$-motive
is split.  In order to obtain information about Chow groups
 from Morava $K$-theories we use operations.
 
The first author constructed in \cite{Sech} and \cite{Sech2} generators of all (not necessarily additive) operations 
from the Morava $K$-theory to $\CH^*\otimes \zz_{(p)}$ 
and from the Morava $K$-theory to itself.
The latter allows one to define the gamma 
filtration on the Morava $K$-theory, and it turns out that its $i$-th graded factor maps surjectively onto 
$\CH^i\ot\Zp$ for all $i\le p^n$.
These operations and various their properties are constructed using the classification of operations
given in a series of articles by
Vishik (see \cite{Vish1}, \cite{Vish2}).

Let $X$ be a smooth geometrically cellular variety
such that the pullback map from $K(n)^*(X)$ to $K(n)^*(X_E)$ is an isomorphism,
where $X_E=X\times_F E$ is the base change to a field $E$ for which $X_E$ becomes cellular.
The operations above as well as symmetric operations of Vishik 
allow us to show that there is no $p$-torsion in Chow groups
of $X$ in codimensions up to $\frac{p^n-1}{p-1}$ (Theorem~\ref{th:no_torsion}).
Moreover, we prove that $p$-torsion is finitely generated in Chow groups of codimension up to
$p^n$, and we provide a combinatorial method to estimate this torsion (Theorem~\ref{th:general_gamma_app}).

For quadratic forms from the ideal $I^{m+2}$ of the Witt ring of a field $F$ of characteristic zero
the $K(m)^*$-motive of the corresponding quadric is split as mentioned above. 
Thus, we obtain that there is no torsion in Chow groups of codimensions less than $2^m$
and we also calculate uniform finite upper bounds on the torsion in $\CH^{2^m}$ which do not depend on the quadric (see Theorem~\ref{th:mult_torsion_quadric}).
In this way Morava $K$-theory provides a conceptual explanation of the nature of this torsion.

These results fit well in the quite established history of estimates on torsion of quadrics
obtained among others by Karpenko, Merkurjev and Vishik.
In particular, Karpenko conjectured in \cite[Conjecture~0.1]{Kar91}
that for every integer~$l$ the Chow group $\CH^l$ of an $n$-dimensional quadric over $F$ is torsion-free
whenever $n$ is bigger than some constant which depends only on $l$.
This was confirmed only for $l\le 4$.
 Recall that by the Arason--Pfister Hauptsatz every anisotropic non-zero quadratic
form from $I^m$ has dimension at least $2^m$ and
 therefore, the absence of torsion in Chow groups of small codimensions of corresponding quadrics
can be considered as an instance of the Karpenko conjecture.
Note also that there are examples of quadrics from $I^{m+2}$ having non-trivial torsion in $\CH^{2^m}$.

Finally, we discuss an approach to cohomological invariants which uses an exact sequence~\eqref{voev}
of Voevodsky (see below). This exact sequence involves motivic cohomology of some simplicial varieties.
For example, this sequence was used in \cite{Sem13} to construct an invariant
of degree $5$ modulo $2$ for groups of type $\E_8$ with trivial Rost invariant and to solve a problem posed by Serre.
\end{ntt}

\medskip

{\bf Acknowledgements.} 
We are sincerely grateful to Alexander Vishik for encouragement and interest in the development
of results related to quadrics and for useful comments on the relation between Morava $K$-theories and cohomological invariants.

We would like to thank sincerely Alexey Ananyevskiy, Stefan Gille, Olivier Haution, Nikita Karpenko, Fabien Morel, Alexander Neshitov,
and Maksim Zhykhovich for discussions and e-mail conversations on the subject of the article.
The second author started to work on this subject during his visit to University Paris 13 in 2014. He would like to express
his sincere gratitude to Anne Qu\'eguiner-Mathieu for her hospitality and numerous useful discussions.

\section{Definitions and notation}\label{sec:notation}

In the present article we assume that $F$ is a field of characteristic $0$.
By $\Fsep$ we denote a separable closure of $F$.

Let $G$ be a semisimple linear algebraic group over a field $F$ (see \cite{Springer}, \cite{Inv}).
A $G$-torsor over $F$ is an algebraic variety $P$ equipped with an action of $G$ such that
$P(\Fsep)\ne\emptyset$ and the action of $G(\Fsep)$ on $P(\Fsep)$ is simply transitive.

The set of isomorphism classes of $G$-torsors over $F$ is a pointed set (with the base point given by
the trivial $G$-torsor $G$) which is in natural one-to-one correspondence with the (non-abelian) Galois cohomology set $H_{\et}^1(F,G)$.

Let $A$ be some algebraic structure over $F$ (e.g. an algebra or quadratic space) such that $\Aut(A)$ is an
algebraic group over $F$. Then an algebraic structure $B$ is called a {\it twisted form} of $A$, if over a separable
closure of $F$ the structures $A$ and $B$ are isomorphic. There is a natural bijection between $H_{\et}^1(F,\Aut(A))$ and the set
of isomorphism classes of the twisted forms of $A$.

For example, if $A$ is an octonion algebra over $F$, then
$\Aut(A)$ is a group of type $\G_2$ and $H_{\et}^1(F,\Aut(A))$ is in $1$-to-$1$ correspondence with the twisted forms
of $A$, i.e., with the octonion algebras over $F$ (since any two octonion algebras over $F$ are isomorphic over a
separable closure of $F$ and since any algebra, which is isomorphic to an octonion algebra over a separable closure
of $F$, is an octonion algebra).
                                                                             
By $\qq/\zz(n)$ we denote the Galois-module $\mathrm{colim}\,\mu_l^{\otimes n}$ taken over all $l$ (see \cite[p.~431]{Inv}).

In the article we use notions from the theory of quadratic forms over fields (e.g. Pfister-forms, Witt-ring). We follow \cite{Inv}, \cite{Lam},
and \cite{EKM}. Further, we use the notion of motives; see \cite{Ma68}, \cite{EKM}.

\section{Geometric constructions of cohomological invariants}\label{sec:geom_contsr_coh_inv}

First, we describe several geometric constructions of cohomological invariants of torsors of degree $2$ and $3$.

Let $G$ be a semisimple algebraic group over a field $F$.
In general, a cohomological invariant of $G$-torsors of degree $n$ with values
in a Galois-module $M$ is a transformation of functors $H_{\et}^1(-,G)\to H_{\et}^n(-,M)$ from the category of field extensions of
$F$ to the category of pointed sets (see \cite[31.B]{Inv}).

\begin{ntt}[Tits algebras and the Picard group]\label{sec31}
In his celebrated article \cite{Ti71} Jacques Tits introduced invariants of degree $2$, called nowadays the {\it Tits algebras}.

There exists a construction of Tits algebras based on the Hochschild--Serre spectral sequence.
For a smooth variety $X$ over $F$ one has $$H^p(\Gamma,H_{\et}^q(X_{\sep},\mathcal{G}))\Rightarrow H_{\et}^{p+q}(X,\mathcal{G})$$
where $\Gamma$ is the absolute Galois group, $X_{\sep}=X\times\Fsep$ and $\mathcal{G}$ is an \'etale sheaf. The induced $5$-term exact sequence is
$$0\to H^1(\Gamma,H_{\et}^0(X_{\sep},\mathcal{G}))\to H_{\et}^1(X,\mathcal{G})\to H^0(\Gamma,H_{\et}^1(X_{\sep},\mathcal{G}))\to H^2(\Gamma,H_{\et}^0(X_{\sep},\mathcal{G}))$$

Let $\mathcal{G}=\mathbb{G}_m$ and let $X$ be a smooth projective geometrically irreducible variety. Then
$$H^1(\Gamma,H_{\et}^0(X_{\sep},\mathbb{G}_m))=H^1(\Gamma,\Fsep^\times)=0$$ 
by Hilbert's Theorem 90,
$H_{\et}^1(X,\mathbb{G}_m)=\Pic(X)$, $H^0(\Gamma,H_{\et}^1(X_{\sep},\mathbb{G}_m))=
(\Pic X_{\sep})^{\Gamma}$,
and $H^2(\Gamma,H_{\et}^0(X_{\sep},\mathbb{G}_m))=H^2(\Gamma,\Fsep^\times)=\Br(F)$. Thus, we obtain an exact sequence
\begin{equation}\label{hochsch}
0\to\Pic X\to(\Pic X_{\sep})^\Gamma\xrightarrow{f}\Br(F)
\end{equation}

The map $\Pic X\to (\Pic X_{\sep})^\Gamma$ is the restriction map and the homomorphism $$(\Pic X_{\sep})^\Gamma\xrightarrow{f}\Br(F)$$
was described by Merkurjev and Tignol in \cite[Section~2]{MT95}. If $X$ is the variety of Borel subgroups of a semisimple
algebraic group $G$, then the Picard group of $X_{\sep}$ can be identified with the free abelian group with basis
$\omega_1,\ldots,\omega_n$ consisting of the fundamental weights, i.e. $\Pic X_{\sep}=\Lambda$, where $\Lambda$ denotes the weight lattice. If $\omega_i$ is $\Gamma$-invariant (e.g. if $G$
is of inner type), then
$f(\omega_i)=[A_i]$ is the Brauer class of the Tits algebra of $G$ corresponding to the (fundamental) representation with the highest weight $\omega_i$
(see \cite{MT95} for a general description of the homomorphism $f$).

Moreover, one can continue the exact sequence~\eqref{hochsch}, namely, the sequence
\begin{equation}\label{hochsch2}
0\to\Pic X\to(\Pic X_{\sep})^\Gamma\xrightarrow{f}\Br(F)\to\Br(F(X))
\end{equation}
is exact, where the last map is the restriction homomorphism (see \cite{MT95}).
\end{ntt}

\begin{ntt}[Tits algebras and $K^0$]\label{titsk0}
There is another interpretation of the Tits algebras related to Grothendieck's $K^0$ functor.
Let $G$ be a semisimple algebraic group over $F$ of inner type and let $X$ be the variety of Borel subgroups of $G$.
By Panin \cite{Pa94} the $K^0$-motive of $X$ is isomorphic to a direct sum of $|W|$ motives, where $W$ denotes the Weyl
group of $G$. Denote these motives by $L_w$, $w\in W$.

For $w\in W$ consider $$\rho_w=\sum_{\{\alpha_k\in\Pi\mid w^{-1}(\alpha_k)\in\Phi^-\}}w^{-1}(\omega_k)\in\Lambda,$$
where $\Pi$ is the set of simple roots, $\Phi^-$ is the set of negative roots, and $\Lambda$ is the weight lattice.

Let $\Lambda_r$ be the root lattice and $$\beta\colon\Lambda/\Lambda_r\to\Br(F)$$
be the Tits homomorphism, which sends a fundamental weight $\omega_i$ to $[A_i]$ (see \cite{Ti71}). In particular, the homomorphism $\beta$ is essentially the homomorphism $f$ from Section~\ref{sec31}.
Then over a splitting field $K$ of $G$, the motive $(L_w)_K$ is isomorphic to a Tate motive
 and the restriction
homomorphism $$K^0(L_w)\to K^0((L_w)_K)=\zz$$ is an injection $\zz\to\zz$ given
 by the multiplication by $\ind A_w$,
where $[A_w]=\beta(\rho_w)$. 
In particular, different motives $L_w$ can be parametrized by the Tits algebras.

Moreover, if all Tits algebras
of $G$ are split, then the $K^0$-motive of $X$ is a direct sum of Tate motives over $F$.
\end{ntt}

\begin{ntt}[Tits algebras and simplicial varieties]
Let $Y$ be a smooth irreducible variety over $F$. Consider the \v{C}ech simplicial scheme $\mathcal{X}_Y$ associated with $Y$,
i.e. the simplicial scheme 
$$Y\llarrow Y\times Y\lllarrow Y\times Y\times Y\cdots$$

Then for all $n\ge 2$ there is a long exact sequence of cohomology groups (see \cite[Corollary~2.2]{Ro07} and \cite[Proof of Lemma~6.5]{Vo11}):
\begin{equation}\label{voev}
0\to H_{\mathcal{M}}^{n,n-1}(\mathcal{X}_Y,\qq/\zz)\xrightarrow{g} H^n_{\et}(F,\qq/\zz(n-1))\to H_{\et}^n(F(Y),\qq/\zz(n-1)),
\end{equation}
where $H_{\mathcal{M}}^{n,n-1}$ is the motivic cohomology and the homomorphism $g$ is induced by the change of topology
from Nisnevich to \'etale (note that by \cite[Lemma~7.3]{Vo03} $\mathcal{X}_Y$ is contractible in the \'etale topology).

Let $n=2$ and let $Y$ be the variety of Borel subgroups of a semisimple algebraic group $G$ of inner type. Then $H_{\et}^2(F,\qq/\zz(1))=\Br(F)$
and we have a long exact sequence
$$0\to H_{\mathcal{M}}^{2,1}(\mathcal{X}_Y)\xrightarrow{g}\Br(F)\to\Br(F(Y))$$
Thus, combining this exact sequence with exact sequence~\eqref{hochsch2} and using explicit description of the homomorphism $f$ from Section~\ref{sec31}, we obtain that
$H_{\mathcal{M}}^{2,1}(\mathcal{X}_Y)=\Lambda/\Lambda'$, where $\Lambda'$ denotes the kernel of $f$. Note also that $\Lambda_r\subset\Lambda'$.
Thus, the Tits homomorphism $\beta$ factors through $H_{\mathcal{M}}^{2,1}(\mathcal{X}_Y)$ by means of the homomorphism $g$.
This gives one more interpretation of the Tits algebras via a change of topology.
\end{ntt}

\begin{ntt}[Rost invariant]
If $G$ is a simple simply connected
algebraic group, then there exists an invariant $$H_{\et}^1(-,G)\to H_{\et}^3(-,\qq/\zz(2))$$
of degree $3$ of $G$-torsors which is called the {\it Rost invariant}
(see \cite{GMS03}). In a particular case when $G$ is the spinor group, 
this invariant is called the {\it Arason invariant}.

If $G$ is of inner type, the Rost invariant can be constructed as follows.
Let $Y$ be a $G$-torsor. Then there is a long exact sequence (see \cite[Section~9]{GMS03})
\begin{multline}\label{rostseq}
0\to A^1(Y,K_2)\to A^1(Y_{\sep},K_2)^\Gamma\\
\xrightarrow{h}\Ker\big(H^3_{\et}(F,\qq/\zz(2))\to H^3_{\et}(F(Y),\qq/\zz(2))\big)\to\CH^2(Y)
\end{multline}
where $A^1(-,K_2)$ is the $K$-cohomology group (see \cite{Ro96}, \cite[Section~4]{GMS03}), $\Gamma$ is the absolute Galois group,
and $Y_{\sep}=Y\times_F F_{\sep}$.
Moreover, $A^1(Y_{\sep},K_2)^\Gamma=\zz$ and $\CH^2(Y)=0$. The Rost invariant of $Y$ is the image of $1\in A^1(Y_{\sep},K_2)^\Gamma$
under the homomorphism $h$. We remark that sequence~\eqref{rostseq} for the Rost invariant is analogous to the
sequence~\eqref{hochsch2} for the Tits algebras arising from the Hochschild--Serre spectral sequence.

We remark also that if $G$ is a group of inner type with trivial Tits algebras (simply-connected or not), then there is a well-defined Rost invariant
of $G$ itself (not of $G$-torsors); see \cite[Section~2]{GP07}.
\end{ntt}

\section{Oriented cohomology theories and the Morava $K$-theory}\label{morava}

In this section we will introduce a cohomology theory --- the Morava $K$-theory. We will prove later that it detects the triviality of some
cohomological invariants (in particular, of the Rost invariant) of algebraic groups.

\begin{ntt}[Characteristic numbers]
Let $X$ be a smooth projective irreducible variety over a field $F$.
Given a partition
$J=(l_1,\ldots,l_r)$ of 
 length $r\ge 0$ with $l_1\ge l_2\ge\ldots\ge l_r>0$
one can associate with it a {\it characteristic class}
$$c_J(X)\in\CH^{|J|}(X)\qquad (|J|=\sum_{i\ge 1} l_i)$$ of $X$ as follows.
Let $P_J(x_1,\ldots,x_r)$ be the smallest symmetric polynomial (i.e., with a minimal number
of non-zero coefficients) containing
the monomial $x_1^{l_1}\ldots x_r^{l_r}$.
We can express $P_J$ as a polynomial on the standard symmetric functions
$\sigma_1,\ldots,\sigma_r$ as
$$P_J(x_1,\ldots,x_r)=Q_J(\sigma_1,\ldots,\sigma_r)$$ for some polynomial
$Q_J$. Let $c_i=c_i(-T_X)$ denote the $i$-th Chern class 
of the virtual normal bundle of $X$. Then
$$c_J(X)=Q_J(c_1,\ldots,c_r).$$
For $|J|=\dim(X)$, the degrees of the characteristic classes are called the {\it characteristic numbers}.

If $J=(1,\ldots,1)$ ($i$ times), then $c_J(X)=c_i(-T_X)$
is the usual Chern class.
If ${\dim X=p^n-1}$ and $J=(p^n-1)$ for some prime number $p$, we write ${c_J(X)=S_{p^n-1}(X)}$.
The degree of the class $S_{\dim X}(X)$ is always divisible by $p$ and we set 
$${s_{\dim X}(X)=\frac{\deg S_{\dim X}(X)}{p}}$$ and
call it the {\it Milnor number} of $X$ (see \cite[Section~4.4.4]{LM}, \cite[Section~2]{Sem13}).

\begin{dfn}
Let $p$ be a prime.
A smooth projective variety $X$ is called a $\nu_n$-variety if
$\dim X=p^n-1$, all characteristic numbers of $X$ are divisible by $p$
and $s_{\dim X}(X)\ne 0\text{ mod }p$.
\end{dfn}
\end{ntt}

\begin{ntt}[Oriented cohomology theories and Borel--Moore homology theories]\label{moore}

In this article we consider oriented cohomology theories $A^*$ in the sense of Levine--Morel (see \cite[Definition~1.1.2]{LM}). By a variety we always mean a quasi-projective variety.

For a smooth variety $X$ over $F$ with the irreducible components $X_1,\ldots,X_l$ we set $A_*(X):=\oplus_{i=1}^{l} A^{\dim X_i-*}(X_i)$. Then the assignment $X\mapsto A_*(X)$ defines an oriented Borel--Moore homology theory in the sense of Levine--Morel (see \cite[Definition~5.1.3]{LM}). Moreover, by \cite[Proposition~5.2.1]{LM} this gives a one-to-one correspondence between oriented cohomology theories and oriented Borel--Moore homology theories on the category of smooth varieties over $F$.

Given an oriented Borel--Moore homology theory on the category of smooth varieties over $F$ we extend it to all separated schemes of finite type over $F$ via $$A_*(Y):=\mathrm{colim}_{V\to Y}A_*(V),$$ where the colimit runs over all projective morphisms $V\to Y$, where $V$ are smooth varieties over $F$, and with push-forward maps as transition maps.

For an oriented Borel--Moore homology theory $A_*$ we say that it satisfies the {\it localization axiom}, if for every quasi-projective $F$-scheme $X$ and a closed $F$-embedding $j\colon Z\to X$ with the open complement $i\colon U\to X$ the sequence
$$A_*(Z)\xrightarrow{j_*}A_*(X)\xrightarrow{i^*} A_*(U)\to 0$$
is exact.
\end{ntt}

\begin{ntt}[Free theories]\label{freetheory}
Consider the algebraic cobordism $\Omega^*$ of Levine--Morel (see \cite{LM}).
By \cite[Theorem~1.2.6]{LM} the algebraic cobordism is a universal oriented cohomology theory, i.e. there is a
(unique) morphism of theories $\Omega^*\to A^*$ for every oriented cohomology theory $A^*$ in the sense of Levine--Morel.

Each oriented cohomology theory $A^*$ is equipped with a $1$-dimensional commutative formal group 
law $\mathcal{F}_A$. For $\Omega^*$ the respective formal group law $\mathcal{F}_{\Omega}$ is the universal one, and the canonical morphism $\mathbb{L}\to\Omega^*(\Spec F)$ from the Lazard ring is an isomorphism (see \cite[Theorem~1.2.7]{LM}).

In this article, when necessary, we consider our oriented cohomology theories as Borel--Moore homology theories and extend them to all separated schemes of finite type over $F$ as in Section~\ref{moore}. Conversely, every oriented Borel--Moore homology theory restricted to the category of smooth varieties gives an oriented cohomology theory.

\begin{dfn}[Levine--Morel,{ \cite[Remark~2.4.14(2)]{LM}}]
Let $R$ be a commutative ring, let $\mathcal{F}_R$ be a formal group law over $R$, and
let $\mathbb{L}\rightarrow R$ be the respective ring morphism.
Then $\Omega_*\otimes_\mathbb{L} R$ is an oriented Borel--Moore homology theory which is called a {\it free theory}.
Its ring of coefficients is $R$, and its associated formal group law is $\mathcal{F}_R$.
\end{dfn}

For example, the Chow theory
is a free theory with the additive formal group law and with the coefficient ring $\zz$ (see \cite[Theorem~1.2.19]{LM}).
In this article $K^0$ stands for a free theory with the multiplicative formal group law and with the coefficient ring $\zz$. If $X$ is a smooth variety over $F$, then $K^0(X)$ is Grothendieck's $K^0$-theory of locally free coherent sheaves on $X$ (see \cite[Theorem~1.2.18]{LM}).

By \cite[Corollary~4.4.3]{LM}
every free theory $A_*$ is
{\it generically constant},
i.e. for every integral scheme $X$ over $F$ the canonical map
$$ A_*(\Spec F) \rarr A_*(\Spec F(X)):= \mathrm{colim}_{U\subset X} A_{*+\dim X}(U)$$
is an isomorphism,
where the colimit is taken over all non-empty open subschemes of~$X$.

By \cite[Theorem~3.2.7]{LM} the algebraic cobordism theory satisfies the localization axiom. Hence, every free theory satisfies the localization axiom as well.

In \cite[Definition~4.1]{Vish1} Vishik defines {\it theories of rational type} in geometric terms
 and proves in
\cite[Proposition~4.7]{Vish1}
that the generically constant
theories of rational type are precisely the free theories.
Vishik's definition allows to describe efficiently the sets of operations between such theories and Riemann--Roch type results for them.
\end{ntt}

\begin{ntt}[Brown--Peterson cohomology and Morava $K$-theories]\label{sec:BP}
For a prime number $p$ and a positive integer $n$ we consider the $n$-th Morava $K$-theory $K(n)^*$ with respect to $p$.
Note that we do not include $p$ in the notation.
We define this theory as a free theory
with the coefficient ring $\zz_{(p)}[v_n,v_n^{-1}]$ where $\deg v_n=-(p^n-1)$
and with a formal group law which we will describe below.

The variable $v_n$, as it is invertible, does not play an important role in computations with Morava $K$-theories,
and sometimes we will prefer to set it to be equal to $1$. 
 It will be always clear from the context
  which $n$-th Morava $K$-theory we use,
 i.e. with $\Zp[v_n, v_n^{-1}]$- or $\Zp$-coefficients.

We follow \cite{Haz} and \cite{Rav}. 
There exists a universal $p$-typical formal group law $\mathcal{F}_{BP}$ over a ring $BP$.
The latter ring is non-canonically isomorphic to the ring $\zz_{(p)}[v_1,v_2,\ldots]$, and from now on we choose the isomorphism defined by Hazewinkel (see \cite[Appendix~2]{Rav}). 
The canonical morphism from $\mathcal{F}_{\Omega}$
over $\LL_{(p)}=\LL\otimes\zz_{(p)}$
to $\mathcal{F}_{BP}$
over the ring $\zz_{(p)}[v_1,v_2,\ldots]$
defines a multiplicative projector on $\Omega^*_{(p)}:=\Omega^*\ot\Zp$ whose image is 
the {\it Brown--Peterson cohomology} $BP^*$. 

The logarithm of the formal group law of the Brown--Peterson theory
equals
$$l(t)=\sum_{i\ge 0}m_it^{p^i},$$
where $m_0=1$ and the remaining variables $m_i$ are related to $v_j$ following Hazewinkel as follows:
$$m_j=\frac 1p\cdot\big(v_j+\sum_{i=1}^{j-1}m_iv_{j-i}^{p^i}\big),$$
see e.g. \cite[Appendix~2.2.1]{Rav}.
Let $e(t)$ be the compositional inverse of $l(t)$.
Then the Brown--Peterson formal group law is given by $e(l(x)+l(y))$. We remark that the coefficients of the logarithm $l(t)$
lie in $\qq[v_1,v_2,\ldots]$, but the coefficient ring of $BP^*$ is $BP=\zz_{(p)}[v_1,v_2,\ldots]$.
Note also that $\deg v_i=-(p^i-1)$. 

We define {an} $n$-th Morava $K$-theory
$K(n)^*$ as a free theory with a $p^n$-typical formal group law $\mathcal{F}$ over $\zz_{(p)}[v_n,v_n^{-1}]$ (or over $\Zp$)
such that the height of $\mathcal{F}$ modulo $p$ is $n$
(see \cite[Definition~3.9]{Sech2}). Thus, even for a fixed prime $p$ and a fixed height $n$ there exist non-isomorphic $n$-th Morava $K$-theories (which are though isomorphic as presheaves of abelian groups, see \cite[Theorem~5.3]{Sech2}).

As in topology we denote by $K(0)^*$ the theory $\CH^*\otimes\QQ$ (independently of a prime~$p$).

In the classical construction of the $n$-th Morava formal group law one takes the $BP$ formal group law and sends all $v_j$ with $j\ne n$
to zero. Modulo the ideal $J$ generated by $p, x^{p^n}, y^{p^n}$ the formal group law for the $n$-th Morava
$K$-theory equals then
$$\mathcal{F}_{K(n)}(x,y)=x+y-v_n\sum_{i=1}^{p-1}\frac 1p\binom pi x^{ip^{n-1}}y^{(p-i)p^{n-1}}\mod J$$
and the logarithm of the corresponding particular $n$-th Morava $K$-theory equals
$$\log_{K(n)}(t)=\sum_{i=0}^{\infty}\frac{1}{p^i}v_n^{\frac{p^{in}-1}{p^n-1}}t^{p^{in}}.$$
More generally, every $n$-th Morava $K$-theory is obtained from $BP^*$ by sending all $v_j$ with $n\!\nmid\!j$ to zero, but $v_j$ with $n|j$ are sent to some multiples of the corresponding
powers of $v_n$ (and the set of all thus obtained theories is independent of the choice of variables $v_j$).

For a variety $X$ over $F$ one has $$K(n)^*(X)=\Omega^*(X)\otimes_\laz\zz_{(p)}[v_n,v_n^{-1}],$$
and $v_n$ is a $\nu_n$-element in the Lazard ring $\mathbb{L}$.

We remark that classically in topology one considers the Morava
$K$-theory with the coefficient ring $\ff_p[v_n,v_n^{-1}]$, but in the present article it is crucial that
we consider an integral version. Note also that as was mentioned above two $n$-th Morava
$K$-theories are additively isomorphic, but are in general not multiplicatively isomorphic.

If $n=1$, for a particular choice of $K(1)^*$ there exists a functorial (with respect to pullbacks) isomorphism of algebras
$K(1)^*(X)/(v_1-1)\simeq K^0(X)\otimes\zz_{(p)}$, which can be 
obtained with the help of the Artin--Hasse exponent (for the latter see \cite[Chapter~7, Section~2]{Robert}).
\end{ntt}

\begin{ntt}[Euler characteristic]\label{eulerc}
The Euler characteristic of a smooth projective irreducible variety $X$ with respect to an oriented cohomology
theory $A^*$ ({\it $A^*$-Euler characteristic})
is defined as the push-forward
$$\pi_*^A(1_X)\in A^*(\Spec F)$$ of the structural morphism $\pi\colon X\to\Spec F$.
E.g., for $A^*=K^0\otimes \zz[v_1,v_1^{-1}]$ with push-forwards defined
as in \cite[Example 1.1.5]{LM}
the Euler characteristic of $X$ equals $$v_1^{\dim X}\cdot\sum(-1)^i\dim H^i(X,\mathcal{O}_X),$$
see \cite[Ch.~15]{Ful}. If $X$ is geometrically irreducible and geometrically cellular, then this element equals $v_1^{\dim X}$
(see \cite[Example~3.6]{Za10}).

For the Morava $K$-theory $K(n)^*$ and a smooth projective irreducible variety $X$ of dimension $d=p^n-1$
the Euler characteristic modulo $p$ equals the element $v_n\cdot u\cdot s_d$ for some $u\in\F{p}^\times$,
where $s_d$ is the Milnor number of $X$ (see \cite[Proposition~4.4.22(3)]{LM}). In particular, it is invertible, if $X$
is a $\nu_n$-variety.
If $\dim X$ is not divisible by $p^n-1$, then the Euler characteristic of $X$ equals zero,
since the target graded ring has non-trivial components only in degrees divisible by $p^n-1$.
\end{ntt}

\begin{ntt}[Motives]
For a theory $A^*$ we consider the category of $A^*$-motives over $F$,
which is defined in the same way as the category of Grothendieck's Chow motives with $\CH^*$
replaced by
$A^*$ (see \cite{Ma68}, \cite{EKM}). Namely, the morphisms between two smooth
projective irreducible varieties $X$ and $Y$ over $F$ are given by
$A^{\dim Y}(X\times Y)$.

By $\ta(l)$, $l\ge 0$, we denote the Tate motives in the category of $A^*$-motives. They are defined in the same way as the Tate motives in the category of Chow motives. Namely, the $A^*$-motive of the projective line splits as a direct sum of the $A^*$-motive of $\Spec F$, which we denote by $\ta$, and another motive, which we denote by $\ta(1)$. Then $\ta(l)$ is defined as $\ta(1)^{\otimes l}$ for $l\ge 0$.

\begin{dfn}
For an oriented cohomology theory $A^*$ and a motive $M$ in the category of $A^*$-motives over $F$ we say that
$M$ is {\it split}, if it is a finite direct sum
of Tate motives over $F$.
\end{dfn}
Note that this property depends on the theory $A^*$, i.e., there exist
smooth projective varieties whose motives are split for some oriented cohomology theories, but not for all
oriented cohomology theories. For example, it follows from Proposition~\ref{pro62} below that the $n$-th Morava $K$-theory
$K(n)^*$ for $p=2$ of an anisotropic $m$-fold Pfister quadric over $F$ is split, if $n<m-1$. On the other hand,
the Chow motive of an anisotropic Pfister quadric is never split.
\end{ntt}

\begin{ntt}[Rost nilpotence for oriented cohomology theories]\label{rostnil}
Let $A^*$ be an oriented cohomology theory and consider the category of $A^*$-motives over $F$.
Let $M$ be an $A^*$-motive over $F$. We say that the Rost nilpotence principle holds for $M$,
if the kernel of the restriction homomorphism $$\End(M)\to\End(M_E)$$ consists of nilpotent correspondences for all
field extensions $E/F$.

By \cite[Section~8]{CGM05} Rost nilpotence holds for Chow motives of all twisted flag varieties.

Rost nilpotence is a tool which allows to descend motivic decompositions over $E$ to motivic decompositions over
the base field $F$. E.g., assume that Rost nilpotence holds for $M$ and that we are given a decomposition
$M_E\simeq\oplus{M_i}$ over $E$ into a finite direct sum. The motives $M$ and $M_i$ are defined as pairs $(X,\rho)$ and $(X_E,\rho_i)$,
where $X$ is a smooth projective variety over $F$, $\rho\in A^*(X\times X)$ and $\rho_i\in A^*(X_E\times X_E)$ are some projectors. 
Assume further that all $\rho_i$ are defined over $F$, i.e. there exist $\eta_i\in A^*(X\times X)$ such that $(\eta_i)_E=\rho_i$. We would like to modify $\eta_i$ to make it a projector, while at the moment we only know that the difference $\eta_i^{\circ 2}-\eta_i$ is in the kernel of the map to $A^*(X_E\times X_E)$ and, thus, is nilpotent. In fact, considering a {\sl commutative} subring of $A^*(X \times X)$ generated by $\eta_i$ for a particular index $i$, one can show that some power of the element $\eta_i$ is a projector. It follows then that $M\simeq\oplus N_i$ for some motives $N_i$ over $F$, and the scalar extension $(N_i)_E$ is isomorphic to $M_i$
for every $i$ (for more details see \cite[Section~8]{CGM05}, \cite[Section~2]{PSZ08}).

Let $M$ be a Chow motive. By \cite[Section~2]{ViYa07} there is a unique lift of the motive $M$ to the category of $\Omega^*$-motives and,
since $\Omega^*$ is the universal oriented cohomology theory, there is a respective motive in the category of $A^*$-motives for every
oriented cohomology theory $A^*$. We denote this $A^*$-motive by $M^A$.

By \cite[Corollary~4.5]{GV18} if $M=(X,\pi)$ is a direct summand of the Chow motive of a twisted flag variety,
then Rost nilpotence holds for $M^A$
for every oriented cohomology theory obtained from $\Omega^*$ by a change of coefficients.

\end{ntt}

\begin{ntt}[Generalized Riemann--Roch theorem]\label{riero}
We follow \cite{Vish2}.
Let $A^*$ be a theory of rational type, let $B^*$ be an oriented cohomology theory
and let $\phi\colon A^*\rarr B^*$ be an operation (which does not necessarily preserve the grading
and does not have to be additive).

For a smooth variety $Z$ over a field $F$ and any $c\ge 0$
denote by $G^c_Z$ the composition
\begin{align*} 
A^*(Z) \rarr & A^*(Z)[[z_1^A,\ldots, z_c^A]]
\xrarr{\phi_{Z\times (\mathbb{P}^\infty)^{\times c}}} 
B^*(Z)[[z_1^B,\ldots, z_c^B]]\\
\alpha \mapsto  &\alpha\cdot z_1^A\cdots z_c^A
\end{align*}
where we have identified $A^*(Z)[[z_1^A,\ldots, z_c^A]]$
with $A^*(Z\times (\mathbb{P}^\infty)^{\times c})$ 
and similarly for $B^*$, 
i.e. $z_i$ is the first Chern class 
of the pullback along the projection of the canonical line bundle $\mathcal{O}(1)$
over the $i$-th product component of $(\mathbb{P}^\infty)^{\times c}$.

Note that by the so-called ``continuity of operations'' (\cite[Proposition~5.3]{Vish2})
for every $c$ and $Z$  the series $G^c_Z(1_Z)$ is divisible by $z_1^B\cdots z_c^B$.
We denote the quotient by $F^c_Z$ and set $F^c=F^c_{\mathrm{pt}}(1)\in B[[z_1^B,\ldots,z_c^B]]$ (we denote $B=B^*(\mathrm{pt})$).
We write $G^c_Z(\alpha)|_{z^B_i=y_i}$ when we plug in nilpotent elements $y_i\in B^*(Z)$
in this series (similarly, for $F^c_Z$ 
and $F^c$).

Finally, denote by $\omega^B_t\in B[[t]]dt$
the canonical invariant $1$-form of the formal group law $F_B$ such that $\omega^B(0)=dt$ (see \cite[Section~7.1]{Vish1}).

The following proposition is a particular case of a general form of Riemann--Roch type theorems {\cite[Theorem~5.19]{Vish2}}.

\begin{prop}[Vishik]\label{th:riemann_roch}
Let $X$ be a smooth variety over a field $F$. Let $i\colon Z\hookrightarrow X$ be a closed embedding of
a smooth subvariety of codimension $c$.
Let $\alpha\in A^*(Z)$ and 
denote by $\mu_1,\ldots, \mu_c$ the $B$-roots of the normal bundle $N_{Z/X}$.

Let $k\ge 0$ and let $L_i$ be line bundles over $Z$ for $1\le i\le k$.
Denote by $x_i=c_1^A(L_i)$, $y_i=c_1^B(L_i)$ their first Chern classes.

Then 
$$ \phi\left(i_* (\alpha \prod_{i=1}^k x_i)\right)
= i_* \mathrm{Res}_{t=0} \frac{G^{c+k}_Z(\alpha)|_{z_i^B=t+_B\mu_i, 1\le i\le c,\  z^B_{c+j} = y_j,
 1\le j\le k}}
{t\cdot \prod_{i=1}^c (t+_B\mu_i)} \omega^B_t. $$
\end{prop}

We will need only the following  instance of this proposition.

\begin{cor}\label{cr:riemann-roch}
We have $\phi(i_* 1_Z) = i_*(\pi^*(F^c)|_{z_i^B=\mu_i})$, where 
$$\pi^*\colon B[[z_1^B,\ldots,z_c^B]]\rarr B^*(Z)[[z_1^B,\ldots,z_c^B]]$$
 is induced by the pullback of the structure map $\pi\colon Z\rarr\Spec F$.

In particular, the right-hand side depends only on the action of operation $\phi$
on products of projective spaces and the $B$-Chern classes
of the normal bundle of $Z$.
\end{cor}
\begin{proof}
Indeed, $1_Z=\pi^*(1)$ 
and $G^c_Z(\pi^* (1))=\phi(\pi^*(1)z_1^A\cdots z_c^A)=\pi^*(\phi(z_1^A\cdots z_c^A))=\pi^*G^c_{\mathrm{pt}}(1)$.
It follows that $F^c_Z(1_Z)=\pi^* F^c_{\mathrm{pt}}(1)=\pi^* F^c$.

We can rewrite the formula in Proposition~\ref{th:riemann_roch} as
$$ \phi(i_* 1_Z) = i_*\mathrm{Res}_{t=0}F^c_Z(1_Z)|_{z_i=t+_B \mu_i}\frac{\omega^B_t}{t}
\prod_i \frac{t+_B\mu_i}{t+_B\mu_i} .$$
Since $\omega^B_t(0)=dt$, we get the required formula.
\end{proof}
\end{ntt}

\begin{ntt}[Topological filtration on free theories] 

For a free theory $A^*$ and a smooth variety $X$ we define the {\it topological filtration} 
(sometimes referred to as a filtration by codimension of support) 
as the kernel of the restriction maps to open subvarieties which have a complement of codimension bounded below:
\begin{equation}\label{fo415}
\tau^i A^*(X):= \bigcup_{U\subset X:\, \codim_X (X\setminus U) \ge i} \Ker(A^*(X)\rarr A^*(U)).
\end{equation}
Since the restriction maps commute with pullbacks, it is clear that $\tau^i A^*$ is a subpresheaf of $A^*$. 

We denote by $\tilde{A}^*:=\tau^1 A^*$
the subpresheaf consisting of all elements
which vanish in the generic points of varieties.  
Note that
since $A^*$ is generically constant,
for every irreducible variety $X$ we have a canonical splitting of abelian groups:
${A^*(X)=A \oplus \tilde{A}^*(X)}$, where $A$ stands for $A^*(\Spec F)$.

One shows using the localization axiom for free theories (\cite[Theorem~3.2.7]{LM})
that this definition is 
equivalent to the one given using
images of push-forwards 
as in \cite[Section~4.5.2]{LM} (cf. \cite[Proposition~1.17(1)]{Sech2} for a relation between the topological filtration on $\Omega^*$ and on free theories).

There exists a canonical surjective map of $\LL$-modules 
$\rho_\Omega\colon \CH^i\ot \LL \rarr \tau^i \Omega^*/\tau^{i+1} \Omega^*$
(see \cite[Corollary~4.5.8]{LM}),
and we denote by $\rho_A\colon\CH^i\ot_\LL A \rarr \tau^i A^*/\tau^{i+1} A^*$
the map of $A$-modules obtained by the change of coefficients of $\rho_\Omega$
from $\LL$ to $A$.

Besides, we denote $\gr^i_\tau \tilde{A}^j:=\tau^i \tilde{A}^j/\tau^{i+1} \tilde{A}^j$, where $\tau^i\tilde{A}^j$ is defined as in formula~\eqref{fo415} with $A^*$ replaced by $\tilde{A}^j$.
\end{ntt}

\section{Gamma filtration on Morava $K$-theories}\label{sec:morava_gamma_filtration}

\begin{ntt}[Operations from Morava $K$-theories]\label{sec:chern_classes}

In article \cite{Sech2} the first author classified all operations
from the $n$-th Morava $K$-theory to the so called $p^n$-typical oriented theories whose
coefficient ring is a free $\Zp$-module.

We will exploit these operations
only when the target theory is either the $n$-th Morava $K$-theory itself
 or the Chow theory with $p$-local coefficients.
There exist certain generators of the algebra of all operations constructed in \cite{Sech2} which in these cases are denoted by $c_i^{K(n)}$ and $c_i^{\CH}$ respectively,
and we summarize their properties in this section.

In this section we consider Morava $K$-theories with $\Zp$-coefficients, i.e. we set $v_n=1$.
This agrees with \cite{Sech, Sech2}.
This reduction to $\Zp$-coefficients does not break the grading completely.
Namely, one can show the following proposition.

\begin{prop}[{\cite[Proposition~4.1.5]{Sech}} \& {\cite[Proposition~3.15]{Sech2}}]\label{prop:morava_grading}

{\ }

\begin{enumerate}
\item Morava $K$-theories $K(n)^*$ are {$\ZZ/(p^n-1)$-graded}
as presheaves of rings.

\item\label{item:pushforward} The grading is compatible with push-forwards,
i.e. for a projective morphism $f\colon X\rarr Y$ of codimension $c$
the push-forward map increases the grading in Morava $K$-theories by $c$:
 $f_*\colon K(n)^i(X)\rarr K(n)^{i+c}(Y)$.

\noindent 
In particular, the first Chern class
of any line bundle $L$ over a smooth variety $X$ lies in $K(n)^1(X)$.

\item The topological filtration on the graded component of the $n$-th Morava $K$-theory
changes only every $p^n-1$ steps, i.e.
we have
$$ \tau^{j+s(p^n-1)+1} \tilde{K}(n)^j = \tau^{j+s(p^n-1)+2} \tilde{K}(n)^j = \ldots =
\tau^{j+(s+1)(p^n-1)} \tilde{K}(n)^j,$$
where $j\in [1, p^n-1]$, $s\ge 0$.

In particular, $\gr^j_\tau \tilde{K}(n)^* = \tilde{K}(n)^j/\tau^{j+p^n-1} \tilde{K}(n)^j$
for $j\colon 1\le j\le p^n-1$.
\end{enumerate}
\end{prop}

We denote the graded components of $K(n)^*$ as $K(n)^1$, $K(n)^2$, $\ldots$, $K(n)^{p^n-1}$
and freely use the notation $K(n)^i$, $K(n)^{i \mod{p^n-1}}$, $K(n)^{i+r(p^n-1)}$
to denote the component $K(n)^j$ where $j\equiv i \mod p^n-1$, $1\le j\le p^n-1$.

\begin{thm}[{\cite[Theorem~4.2.1]{Sech}, \cite[Theorem~3.16]{Sech2}}]\label{th:chern_classes}
There exist operations $c_i^{K(n)}\colon K(n)^{i \mod p^n-1}\rarr K(n)^{i\mod p^n-1}$
and $c_i^{\CH}\colon K(n)^i\rarr \CH^i\ot\Zp$ for $i\in\ZZ^{>0}$
satisfying the following properties (we omit the index $K(n)$ resp. $\CH$ in the notation of $c_i$,
since the index is always clear from the context):
\begin{enumerate}
\item For any smooth variety $X$ and for every pair of elements $x,y\in K(n)^*(X)$ the modified Cartan's formula holds:
$$ c_{tot}(x+y)=\mathcal{F}_{K(n)}(c_{tot}(x), c_{tot}(y)),$$
where $\mathcal{F}_{K(n)}$ is the formal group law for the Morava $K$-theory, $c_{tot}=\sum_{i\ge 1} c_it^i$,
$t$ is a formal variable
and we naturally consider each operation $c_j$ 
to be defined on the whole group $K(n)^*$ via the composition with the natural projection to $K(n)^j$. 

\noindent
The equality takes place in $K(n)^*(X)\ot_{\Zp} \Zp[[t]]$ or $\CH^*(X)\ot\Zp[[t]]$.
\item Every operation from the presheaf 
$\tilde{K}(n)^*$ 
to the corresponding target theory
can be uniquely expressed as a formal power series in $c_i$'s with $\Zp$-coefficients.
\end{enumerate}
\end{thm}

\begin{ntt}[The gamma filtration]
The above operations from the $n$-th Morava $K$-theory to itself allow to define 
the gamma filtration verbatim as for the $K$-theory.
We recall first the classical picture, since
the situation with Morava $K$-theories is very similar.

Recall that for $K^0$
the Chern classes $c_i\colon K^0\rarr \CH^i$ 
can be restricted to additive maps
$$c_i\colon \mathrm{gr}^i_\tau K^0\rarr \CH^i,$$
where $\mathrm{gr}^i_\tau K^0$ stand for the graded components of the topological filtration $\tau^\bullet$ on $K^0$.

There is also a canonical map 
$(\rho_{K_0})_i\colon\CH^i\rarr \mathrm{gr}^i_\tau K^0$
which sends a cycle $Z$
to the class of the coherent sheaf $[\OO_Z]$.
The compositions $(\rho_{K_0})_i \circ c_i$, $c_i\circ (\rho_{K_0})_i$
are multiplications by $(-1)^{i-1}(i-1)!$.
In particular, $c_i$ is surjective if one inverts $(i-1)!$.

The gamma filtration $\gamma^\bullet$ for $K^0$ is an approximation of the topological filtration. One has $\gamma^i\subset\tau^i$ for all $i$,
and $\gamma^i=\tau^i$ for $i\le 2$. Moreover, the induced map $$\mathrm{gr}^i_\gamma K^0\rarr\mathrm{gr}^i_\tau K^0\xrightarrow{c_i}\CH^i$$
is surjective for $i\le 2$.

A similar picture holds for the Morava $K$-theories.
The canonical additive map $$(\rho_{K(n)})_i\colon\CH^i\otimes\Zp\rarr \tau^iK(n)^*/\tau^{i+1}K(n)^*$$
is defined using \cite[Corollary~4.5.8]{LM}. It is possible to calculate the compositions $(\rho_{K(n)})_i\circ c_i^{\CH}$,
$c_i^{\CH}\circ (\rho_{K(n)})_i$, 
and they turn out to be isomorphisms in a bigger range compared to $K^0$.

\begin{prop}[{\cite[Proposition~6.2]{Sech2}}]\label{prop:chow_top_filt_morava}
The canonical map 
$$(\rho_{K(n)})_i\colon\CH^i\otimes\Zp\rarr \tau^iK(n)^*/\tau^{i+1}K(n)^*$$ 
is an isomorphism for $0\le i\le p^n$, and the map $c_i^{\CH}$ is its inverse for $1\le i\le p^n$.
\end{prop}

In general, it is hard to calculate the topological filtration for $K(n)^*(X)$ even if
$X$ is a geometrically cellular variety and $K(n)^*$-motive of $X$ is split. 
The problem is that the topological filtration is not strictly respected by the base change restrictions
like $K(n)^*(X)\rarr K(n)^*(\overline{X})$.
 The gamma filtration which we will now describe is a computable approximation 
 to the topological filtration which lacks such ``handicap''.

\begin{dfn}[{\cite[Definition~6.1]{Sech2}}]
Define the gamma filtration on $K(n)^*$ of a smooth variety $X$ by the following formulas:
$$ \gamma^0 K(n)^*(X)=K(n)^*(X),$$
$$\gamma^m K(n)^*(X):= \langle c_{i_1}^{K(n)}(\alpha_1)\cdots c_{i_k}^{K(n)}(\alpha_k)\mid \sum_j i_j\ge m, i_j \ge 1, k\ge 1, \alpha_j \in K(n)^*(X)\rangle,$$
where the $\langle\,,\rangle$-brackets denote the generation as $\Zp$-modules and $m\ge 1$.
\end{dfn}

It is clear from the definition that $\gamma^m K(n)^*$ is an ideal subpresheaf of $K(n)^*$.

\begin{thm}[{\cite[Proposition~6.2]{Sech2}}]\label{th:morava_gamma_properties}
The gamma filtration and the topological filtration satisfy the following properties:

\begin{enumerate}[i)]
\item $\gamma^i\subset\tau^i$ for all $i$;
\item $c_i^{\CH}|_{\tau^{i+1}K(n)^*}=0$, $c_i^{\CH}|_{\gamma^{i+1}K(n)^*}=0$;
\item\label{item:gr_chern_isom_rat} the operation $c^{\CH}_i$ is additive when restricted to $\tau^i K(n)^*$ or $\gamma^i K(n)^*$
and the map
$$c^{\CH}_i\otimes\mathrm{id}_\qq\colon \mathrm{gr}^i_\gamma K(n)^*\otimes_{\zz_{(p)}}\qq\rarr \CH^i\ot\QQ$$ is an isomorphism;
\item\label{item:chern_isom_top} $c_i^{\CH}$ 
induces an additive isomorphism between $\gr^i_\tau K(n)^*$
and $\CH^i\ot\Zp$ for ${1\le i\le p^n}$;
\item\label{item:chern_surj} $c_i^{\CH}$ restricted to  $\gamma^iK(n)^*$ is surjective for $1\le i\le p^n$;
\item\label{item:gr_gamma_support} $\mathrm{gr}_\gamma^iK(n)^*=\mathrm{gr}_\gamma^iK(n)^{i\mod p^n-1}$.
\end{enumerate}
\end{thm}

In section~\ref{bounds} we will use the Riemann--Roch formula 
(Proposition~\ref{th:riemann_roch}, Corollary~\ref{cr:riemann-roch})
to perform computations with the gamma filtration.
Let us sketch how it applies.

We follow the notation of Section~\ref{riero}.
Let $\phi \colon A^*\to B^*$ be an operation, let $X$ be a smooth variety, and let $i\colon Z\hookrightarrow X$ be its smooth
closed subvariety of codimension $c$.

It follows from the Riemann--Roch formula that the value
$\phi(i_*1_Z)$ is equal to $b\cdot 1_Z$ modulo $(c + 1)$-st part of the topological filtration, where $b \in B$ is the
coefficient of $z_1^B\cdots z_c^B$ in the series $\phi(z_1^A\cdots z_c^A)$. The following technical statements describe this coefficient for some operations for the Morava $K$-theory.

\begin{prop}[{\cite[Proposition~6.11]{Sech2}}]\label{prop:constant_cpn}
Let $c^{K(n)}_{p^n}$ be the respective operation from $K(n)^1$  to $\gamma^{p^n}K(n)^1$.

Denote by $e_j$, $j\ge 0$, the coefficient
of the monomial $z_1\cdots z_{1+j(p^n-1)}$ in the series
$c^{K(n)}_{p^n}(z_1\cdots z_{1+j(p^n-1)}) 
\in K(n)^1((\mathbb{P}^\infty)^{\times {1+j(p^n-1)}})$.

Then for all primes $p$ and for all $j\ge 1$ we have $e_j\in \Zp^{\times}$.
\end{prop}

\begin{prop}[{\cite[Proposition~6.13]{Sech2}} for $p=2$]\label{prop:constant_chern_classes2}
Let $j\ge 0$.

There exist operations $\chi, \psi:K(n)^1\rarr \gamma^{2^{n+1}-1}K(n)^1$
which satisfy the following. 

Denote by $g_j, f_j \in\Z{2}$, the coefficients
of the monomial $z_1\cdots z_{1+j(2^n-1)}$ in the series
$\chi(z_1\cdots z_{1+j(2^n-1)})$, 
 $\psi(z_1\cdots z_{1+j(2^n-1)}) \in K(n)^1((\mathbb{P}^\infty)^{\times {1+j(2^n-1)}})$, respectively.
Then 

\begin{enumerate}[1)]
\item we have $g_j = f_j = 0$ for $j=0,1$.
\end{enumerate}

Let $j\ge 2$.
\begin{enumerate}[1)]
\setcounter{enumi}{1}
\item\label{item:coef_2pn-1}
We have $g_j\in 2^{t_j}\Z{2}^\times$ where $t_j=\nu_2(j-1)+2$ if $j$ is odd,
and $t_j=1$ if $j$ is even. Here $\nu_2$ denotes the $2$-adic valuation on integers.

\item\label{item:coef_adams} We have $f_j \in 2^{2^n}\Z{2}^{\times}$.
\end{enumerate}
\end{prop}
\end{ntt}

\end{ntt}

\section{Some computations of the Morava $K$-theory}\label{moravaappl}

\begin{dfn}
Let $m\ge 2$ and let $\alpha\in H_{\et}^m(F,\mu_p^{\otimes m})$ be a non-zero pure symbol.
A motive ${R_m=(X,\pi)}$ in the category of Chow motives
with $\zz_{(p)}$-coefficients is called the {\it (generalized) Rost motive} for $\alpha$, if
it is indecomposable, splits as a sum of Tate motives over $F(X)$ and
for every field extension $K/F$ the following conditions are equivalent:
\begin{enumerate}
\item $(R_m)_K$ is decomposable;
\item $(R_m)_K\simeq\bigoplus_{i=0}^{p-1}\zz_{(p)}(b\cdot i)$ with $b=\frac{p^{m-1}-1}{p-1}$;
\item $\alpha_K=0\in H_{\et}^m(K,\mu_p^{\otimes m})$.
\end{enumerate}
\end{dfn}
The fields $K$ from this definition are called splitting fields of $R_m$.

The Rost motives were constructed by Rost and Voevodsky (see \cite{Ro07}, \cite{Vo11}). Namely, for all pure symbols $\alpha$
there exists a smooth geometrically irreducible projective $\nu_{m-1}$-variety $X$ (depending on $\alpha$) over $F$ such that the Chow motive
of $X$ has a direct summand isomorphic to $R_m$ and for every field extension $K/F$ the motive $(R_m)_K$
is decomposable iff $X_K$ has a $0$-cycle of degree coprime to $p$. The variety $X$ is called a {\it norm variety} of $\alpha$.
Moreover, it follows from \cite[Lemma~9.2]{Ya12} that for a given $\alpha$ the respective Rost motive is unique.

E.g., if $p=2$ and $\alpha=(a_1)\cup\ldots\cup(a_m)$ with $a_i\in F^\times$, then one can take for $X$
the projective quadric given by the equation $\langle\!\langle a_1,\ldots,a_{m-1}\rangle\!\rangle\perp\langle -a_m\rangle=0$,
where $\langle\!\langle a_1,\ldots,a_{m-1}\rangle\!\rangle$ denotes the Pfister form.
(We use the standard notation from the quadratic form theory as in \cite{Inv} and \cite{EKM}.)

As in Section~\ref{rostnil} using \cite[Section~2]{ViYa07}
one finds a unique lift of the Rost motive $R_m$ to the category of
$A^*$-motives for every oriented cohomology theory $A^*$.
We will denote this $A^*$-motive by the same letter $R_m$, since $A^*$ will be always clear from the context. Recall that by $\ta(l)$, $l\ge 0$, we denote the Tate motives in the
category of $A^*$-motives. If $A^*=\CH^*\otimes\zz_{(p)}$, we keep the usual notation $\ta(l)=\zz_{(p)}(l)$.

Moreover, it follows from \cite[Lemma~4.2]{GV18} that Rost nilpotence holds for $R_m$ with respect to every free theory $A^*$, since $R_m$ splits over the residue fields of all points of $X$.

\begin{prop}\label{pro62}
Let $p$ be a prime number, let $n\ge 0$ and $m\ge 2$ be integers and $b=\frac{p^{m-1}-1}{p-1}$.
For a non-zero pure symbol $\alpha\in H_{\et}^m(F,\mu_p^{\otimes m})$ consider the respective Rost motive $R_m$. Then
\begin{enumerate}
\item If $n<m-1$, then the $K(n)^*$-motive $R_m$ is a sum of $p$ Tate motives
$\oplus_{i=0}^{p-1}\ta(b\cdot i)$.
\item If $n=m-1$, then the $K(n)^*$-motive $R_m$ is a sum of the Tate motive $\ta$
and an indecomposable motive $L$ such that
\begin{equation}\label{form64}
K(n)^*(L)\simeq(\zz_{(p)}^{\oplus(p-1)}\oplus(\zz/p)^{\oplus{(m-2)(p-1)}})\otimes\zz_{(p)}[v_n,v_n^{-1}].
\end{equation}
For a field extension $K/F$ the motive $L_K$ is isomorphic to a direct sum of Tate motives iff the symbol $\alpha_K=0$. If $p>2$, then this is additionally equivalent to the condition that the motive $L_K$ is decomposable.
\item If $n>m-1$, then the $K(n)^*$-motive $R_m$ is indecomposable and $K(n)^*(R_m)$
is isomorphic to the group
\begin{equation}\label{form63}
\CH^*(R_m)\otimes\zz_{(p)}[v_n,v_n^{-1}]\simeq (\zz_{(p)}^{\oplus p}\oplus(\zz/p)^{\oplus{(m-2)(p-1)}})\otimes\zz_{(p)}[v_n,v_n^{-1}].
\end{equation}
For a field extension $K/F$ the motive $(R_m)_K$ is decomposable iff $\alpha_K=0$.
In this case $(R_m)_K$ is a sum of $p$ Tate motives.
\end{enumerate}
\end{prop}
\begin{proof}
Let $X$ be a norm variety of dimension $p^{m-1}-1$ for $\alpha$.
Denote by $\overline R_m$ the scalar extension of $R_m$ to its splitting field.
By \cite[Theorem~10.6]{Ya12} (cf. \cite[Theorem~3.5, Proposition~4.4]{ViYa07}) the restriction map for the Brown--Peterson theory $BP^*$
\begin{equation}\label{eq1}
\res\colon BP^*(R_m)\to BP^*(\overline R_m)=BP^{\oplus p}
\end{equation}
is injective, and the image equals 
\begin{equation}\label{eq2}
BP^*(R_m)\simeq BP\oplus I(m-1)^{\oplus (p-1)},
\end{equation}
where $I(m-1)$ is the ideal in the ring $BP=\zz_{(p)}[v_1,v_2,\ldots]$ generated by the elements
$\{v_0,v_1,\ldots, v_{m-2}\}$ where $v_0=p$. We remark that the article \cite[Theorem~10.6]{Ya12} deals with the bigraded version of the Brown--Peterson cohomology theory $ABP^{*,*'}$. Nevertheless, due to Levine's comparison result \cite{L05} Yagita identifies $ABP^{2*,*}$ with $BP^*$.

The projectors for the motive $R_m$ lie in the group $K(n)^{p^{m-1}-1}(R_m\otimes R_m)$, and by \cite[Theorem~4.4.7]{LM} the elements of
$K(n)^{p^{m-1}-1}(R_m\otimes R_m)$ are $\Zp$-linear combinations of elements of the form
\begin{equation}\label{summan}
v_n^s\cdot [Y\to X\times X],\quad s\in \ZZ,
\end{equation}
where $Y$ is a resolution of singularities of a closed irreducible subvariety of $X\times X$, and $-s(p^n-1)+\codim Y=p^{m-1}-1$.

(1) Assume first that $n<m-1$.
Since the ideal $I(m-1)$ contains $v_n$ for $n<m-1$ and $v_n$ is invertible in $K(n)$, we immediately get that all
elements in $K(n)^*(\overline R_m)$ are rational, i.e., are defined over the base field.

Therefore, since the motive $R_m$ is geometrically split, all elements in $K(n)^*(\overline R_m\otimes \overline R_m)$ are rational, and hence by Rost nilpotence
for $R_m$ this gives the first statement of the proposition.

(3) Let $n>m-1$. First of all, taking the tensor product $-\otimes_{BP(\Spec F)} K(n)$ with formula~\eqref{eq1} and
using \eqref{eq2} one immediately gets formula~\eqref{form63} for $K(n)^*(R_m)$ and $\CH^*(R_m)$.

We have $\dim X=p^{m-1}-1<p^n-1=-\deg v_n$.

Since every projector in ${K(n)^{p^{m-1}-1}(R_m\otimes R_m)}$ is a linear combination of elements of the form~\eqref{summan} and $\dim (X\times X)=2(p^{m-1}-1)$, we must have $s=0$ in all summands.
Therefore, every projector $\rho$ in $K(n)^{p^{m-1}-1}(R_m\otimes R_m)$ comes from the connective Morava $K$-theory $CK(n)^*$ (the connective Morava $K$-theory is a free oriented cohomology theory with the same formal group law as the Morava $K$-theory, but with the coefficient ring $\Zp[v_n]$). Thus, we have the following commutative diagram
\[
 \xymatrix{
 CK(n)^*(R_m\otimes R_m) \ar@{->}[r]\ar@{->}[d] & CK(n)^*(\overline R_m\otimes \overline R_m) \ar@{^{(}->}[d]\\
 K(n)^*(R_m\otimes R_m) \ar@{->}[r] & K(n)^*(\overline R_m\otimes \overline R_m)
 }
 \]
 and the rational projector $\bar\rho$ comes from some rational projector $\bar\tau\in CK(n)^*(\overline R_m\otimes \overline R_m)$. Note that the right vertical arrow is injective, since $\overline R_m$ is a direct sum of Tate motives.

By \cite[Section~2]{ViYa07} the $CK(n)^*$-motive $R_m$ is indecomposable, since so is the respective Chow motive. Therefore $\bar\tau$ is either zero or the identity projector. Therefore, so is $\bar\rho$ and, hence, by Rost nilpotence for $R_m$ so is the projector $\rho$. Therefore, the $K(n)^*$-motive $R_m$ is indecomposable.

(2) Assume now that $n=m-1$.
Since the $K(n)^*$-Euler characteristic of $X$ equals $u\cdot v_n$ for some
$u\in\zz_{(p)}^\times$ (see Section~\ref{eulerc}), the element ${v_n^{-1}\cdot u^{-1}(1\times 1)\in K(n)^*(R_m\otimes R_m)}$
is a projector defining the Tate motive $\ta$ where $1\times 1 \in K(n)^0(X\times X)$.
Note that this projector lies in $K(n)^*(R_m\otimes R_m)$, 
since this is true over a splitting field of $R_m$ and since $1\times 1$ is a rational element.
Thus, we get the decomposition $R_m\simeq\ta\oplus L$ for some motive $L$.

Taking the tensor product $-\otimes_{BP(\Spec F)} K(n)$ with formula~\eqref{eq1} and
using \eqref{eq2} one immediately gets formula~\eqref{form64} for $K(n)^*(L)$.

We claim now that $L$ is indecomposable. If $p=2$, then this is clear, since in this case $L$ over a splitting field of $R_m$ is a Tate motive. So, we assume that $p>2$.

We have $\dim X=p^{m-1}-1=p^n-1$. Since every projector in $K(n)^{p^n-1}(R_m\otimes R_m)$ is a linear combination of elements of the form~\eqref{summan} and $\dim (X\times X)=2(p^n-1)$, we must have $s=0$, $1$ or $-1$ in all summands.

If a projector contains a summand with $s=-1$, then by dimensional reasons this summand is up to a scalar of the form $v_n^{-1}(1\times 1)$. Subtracting this summand we obtain a rational element, say $\bar\sigma$, in $K(n)^{p^n-1}(\overline R_m\otimes \overline R_m)$ which comes from a rational element in $CK(n)^{p^n-1}(\overline R_m\otimes \overline R_m)$. To prove indecomposability of $L$ it is sufficient to prove its indecomposability modulo $p$.

We denote by $\Ch^*$ the Chow theory modulo $p$. The Chow motive $\overline R_m$ is a direct sum of Tate motives with pairwise distinct twists, the Chow motive $R_m$ is indecomposable over $F$ and some power of any rational cycle in $\End(\overline R_m)$ is a rational projector. Therefore, one can see that the only rational cycles in $\Ch^{p^n-1}(\overline R_m\otimes\overline R_m)$ are scalar multiples of the diagonal.

Thus, by dimensional reasons $\bar\sigma$ is of the form $a\Delta_{\overline R_m}+bv_n(\mathrm{pt}\times\mathrm{pt})$, where $a,b\in\zz/p$, $\Delta_{\overline R_m}$ is the diagonal of $\overline R_m$ and $\mathrm{pt}\times\mathrm{pt}$ is the class of a rational point on $\overline X\times \overline X$.

Therefore, the original rational projector in $K(n)^*(\overline R_m\otimes \overline R_m)$ is modulo $p$ of the form $$a\Delta_{\overline R_m}+bv_n(\mathrm{pt}\times\mathrm{pt})+cv_n^{-1}(1\times 1)$$
for some $c\in\zz/p$.
Composing this elements with itself and using that \begin{align*}
(\mathrm{pt}\times\mathrm{pt})\circ(1\times 1)&=1\times\mathrm{pt},\\
(1\times 1)\circ(\mathrm{pt}\times\mathrm{pt})&=\mathrm{pt}\times 1,\\
(\mathrm{pt}\times\mathrm{pt})\circ(\mathrm{pt}\times\mathrm{pt})&=0,\\
(1\times 1)\circ(1\times 1)&=u\cdot v_n (1\times 1),
\end{align*}
we obtain that this element is a projector only if $(a,b,c)=(0,0,0)$ (the trivial projector), $(a,b,c)=(1,0,0)$ (the diagonal), $(a,b,c)=(0,0,u^{-1})$ (the projector ${v_n^{-1}\cdot u^{-1}(1\times 1)}$), or $(a,b,c)=(1,0,-u^{-1})$ (the complementary projector ${\Delta_{\overline R_m}-v_n^{-1}\cdot u^{-1}(1\times 1)}$).
Thus, the motive $L$ is indecomposable.
\end{proof}

\begin{rem}
Recall that some generalized Rost motives appear as direct summands of motives of some twisted flag varieties (e.g. Pfister quadrics for $p=2$ or varieties of type $\FF_4$ for $(m,p)=(3,3)$ or $(5,2)$ or of type $\E_8$ for $(m,p)=(3,5)$; see \cite{Ro07}, \cite{NSZ09}, \cite{Mac09}, \cite[Section~7]{PSZ08}).
The above proposition demonstrates a difference between $K^0$ and the Morava $K(n)$-theory, when $n>1$. 
By \cite{Pa94} $K^0$
of all twisted flag varieties is $\zz$-torsion-free. This is not the case for $K(n)^*$, $n>1$.

Moreover, the same arguments as in the proof of the proposition
show that the connective $K$-theory $CK(1)^*$ (see \cite{Cai08}) of Rost motives $R_m$ for $m>2$ contains non-trivial $\zz$-torsion.
\end{rem}

\begin{rem}
The same proof shows that the Johnson--Wilson theory $E(n)^*$ of the Rost motive $R_m$ is split, if $n<m-1$. By definition, the coefficient ring of
the Johnson--Wilson theory $E(n)^*$ equals $\zz_{(p)}[v_1,\ldots,v_n][v_n^{-1}]$.
\end{rem}

\begin{rem}
The Chow groups of the Rost motives are known; see \cite[Theorem~5]{Ro90}, \cite[Theorem~8.1]{KM02}, \cite[Theorem~RM.10]{KM13}, \cite[Corollary~10.8]{Ya12},
\cite[Section~4.1]{Vish_Symm_2}.
\end{rem}

The proof of the following proposition is close to \cite[Section~8]{A12}.

\begin{prop}\label{moravaaffine2}
Let $A^*$ be an oriented
generically constant cohomology theory in the sense of Levine--Morel satisfying the localization axiom.
Let $Z$ be a smooth variety over a field $F$.
Assume that there exists a smooth projective variety $Y$ with invertible Euler characteristic
with respect to $A^*$ and such that for every point $y\in Y$ (not necessarily closed)
the natural pullback $$A^*(F(y))\to A^*(Z_{F(y)})$$ is an isomorphism.

Then the pullback of the structural morphism $Z\xrightarrow{\pi}\Spec F$
induces an isomorphism $$A^*(F)\xrightarrow{\sim} A^*(Z).$$
\end{prop}

Before proving this proposition we prove the following lemma.

\begin{lem}\label{lem610}
Let $X$ be a variety over $F$, let $Z$ be a smooth variety over $F$ and let $A^*$ be as in Proposition~\ref{moravaaffine2}.
Assume that the natural pullback $A^*(F(x))\to A^*(Z_{F(x)})$ is an isomorphism
for every point $x\in X$.
Then the pullback $A^*(X)\to A^*(Z\times X)$ of the projection is surjective.
\end{lem}
\begin{proof}
We use the Borel--Moore homology theory associated with $A^*$ as explained
in Section~\ref{moore}.

Let $X_1,\ldots, X_l$ be the irreducible components of $X$ with generic points $x_1,\ldots,x_l$.
We have the following commutative diagram
\[
 \xymatrix{
\bigoplus_{i=1}^{l}\underset{X'\subset X_i}{\mathrm{colim}}\,A_*(X')
\ar@{->}[r]\ar@{->}[d] & A_*(X) \ar@{->}[r]\ar@{->}[d] & \bigoplus_{i=1}^{l}A_*(F(x_i)) \ar@{->}[r] \ar@{->}[d]&0\\
\bigoplus_{i=1}^{l}\underset{X'\subset X_i}{\mathrm{colim}}\,A_*(Z\times X')
\ar@{->}[r] & A_*(Z\times X) \ar@{->}[r] & \bigoplus_{i=1}^{l}A_*(Z_{F(x_i)}) \ar@{->}[r] &0
 }
 \]
where the vertical arrows are pullbacks of the respective projections, 
the colimits are taken over all closed codimension $\ge 1$
subvarieties of irreducible components of $X$, 
and the rows are exact by the localization property.

By the assumptions the right vertical arrow is an isomorphism.
Note that every closed subvariety $X'$ of $X$ satisfies the assumption of the lemma. 
Therefore, we can argue by induction on the dimension of varieties $X'$ that the left vertical arrow is surjective.
It follows by a diagram chase that the middle vertical arrow is surjective as well.
\end{proof}

\begin{proof}[Proof of Proposition~\ref{moravaaffine2}]
We omit gradings in the proof.

Let $a\colon Y\to\Spec F$ be the structural morphism, let $b\colon Z\times Y\to Y$ and $c\colon Z\times Y\to Z$
be the projections. Consider now the following commutative diagram:
$$
\xymatrix{
A(F) \ar@{->}[dd]^{\simeq}\ar@{->}[rd]^-{a^*}\ar@{->}[rrr]^-{\pi^*}& & & A(Z) \ar@{->}[ld]^-{c^*} \ar@{->}[dd]^{\simeq}\\
& A(Y) \ar@{->}[r]^-{b^*}\ar@{->}[ld]^-{a_*} & A(Z\times Y)\ar@{->}[rd]^-{c_*} &\\
A(F)\ar@{->}[rrr]^-{\pi^*} & & & A(Z)
}
$$
By Lemma~\ref{lem610} applied to the variety $Y$ the homomorphism $b^*$ is surjective. The left and the right vertical arrows
are isomorphisms, since they are multiplications by the $A^*$-Euler characteristic of $Y$ which is invertible.

Therefore, by a diagram chase the bottom horizontal arrow is surjective. But $A(F)$ is a direct summand of
$A(Z)$ because the theory $A^*$ is generically constant. Therefore, the bottom arrow is an isomorphism.
\end{proof}

Let $(a_1)\cup\ldots\cup(a_m)\in H_{\et}^m(F,\zz/2)$ be a pure symbol, $a_i\in F^\times$.
The quadratic form $q=\langle\!\langle a_1,\ldots,a_{m-1}\rangle\!\rangle\perp\langle -a_m\rangle$
is called a {\it norm form} and the respective projective quadric given by $q=0$ is called a (projective)
{\it norm quadric}. The respective {\it affine norm quadric} is an open subvariety of the projective norm quadric
given by the equation $$\langle\!\langle a_1,\ldots,a_{m-1}\rangle\!\rangle=a_m,$$ i.e. setting the last coordinate to $1$.

\begin{cor}\label{moravaaffine}
Let $0\le n<m-1$ and set $p=2$. Consider the affine norm quadric $X^{\aff}$ of dimension $2^{m-1}-1$ corresponding
to a pure symbol in $H_{\et}^m(F,\zz/2)$. Then the pullback of the structural morphism $X^{\aff}\xrightarrow{\pi}\Spec F$
induces an isomorphism $$K(n)^*(X^{\aff})=K(n)^*(F).$$
\end{cor}
\begin{proof}
Let $\alpha:=(a_1)\cup\ldots\cup(a_m)\in H_{\et}^m(F,\zz/2)$ be our pure symbol, $a_i\in F^\times$,
$q$ the norm form for $\alpha$, and $Q$ the respective projective norm quadric given by $q=0$.
Let $Y$ be the projective norm quadric of dimension $2^n-1$ corresponding to the subsymbol
$$(a_1)\cup\ldots\cup(a_{n+1})\in H_{\et}^{n+1}(F,\zz/2).$$

We need to check the conditions of Proposition~\ref{moravaaffine2}.
By the choice of $Y$ it is a $\nu_n$-variety (see e.g. \cite[Section~2]{Sem13}). Therefore, by Section~\ref{eulerc}
its $K(n)^*$-Euler characteristic is invertible.

Moreover, the quadratic form $q$ is split completely over $F(y)$ for any point $y$ of $Y$.
In particular, $X^{\aff}_{F(y)}$ is a split odd-dimensional affine quadric.
The complement $Q':=Q\setminus X^{\aff}$
is the projective Pfister quadric $\langle\!\langle a_1,\ldots,a_{m-1}\rangle\!\rangle=0$ of dimension $2^{m-1}-2$, and both $Q$ and $Q'$ are
split over $F(y)$.

Let $W$ be a split affine quadric of odd dimension $2k-1\ge 1$.
Then it is well known that $W$ is given by the equation $\sum_{i=1}^k x_iy_i=1$ in the affine space $\mathbb{A}^k\times \mathbb{A}^k$, and the projection $W\to\mathbb{A}^k\setminus\{0\}$, $(x,y)\mapsto y$, is a rank $(k-1)$ affine bundle over $\mathbb{A}^k\setminus\{0\}$.
Therefore, by homotopy invariance $K(n)^*(X^{\aff}_{F(y)})=K(n)^*(W)=K(n)^*(\mathbb{A}^k\setminus\{0\})=K(n)^*(F(y))$ with $k=2^{m-2}$.
We are done.
\end{proof}

\begin{rem}
In the proof of Corollary~\ref{moravaaffine}
the motive of $W$ in the category $DM$ of Voevodsky (see \cite[Lecture~14]{MVW}) is isomorphic by homotopy invariance to the motive of $\mathbb{A}^k\setminus\{0\}$.
The Gysin exact triangle \cite[14.5.5]{MVW} immediately implies that the motive of $\mathbb{A}^k\setminus\{0\}$ is isomorphic to $\zz\oplus\zz(k)[2k-1]$.
In particular, the motive of $X^{\aff}$ in the category $DM$ of motives of Voevodsky is not a Tate motive even if the base field is algebraically closed.
\end{rem}

Let now $B$ be a central simple $F$-algebra of a prime degree $p$ and $c\in F^\times$.
Consider the Merkurjev--Suslin variety
$$\MS(B,c)=\{\alpha\in B\mid \Nrd(\alpha)=c\},$$
where $\Nrd$ stands for the reduced norm on $B$.

\begin{cor}
In the above notation the structural morphism induces an isomorphism $$A^*(\MS(B,c))\simeq A^*(F),$$ when $A$ is Grothendieck's $K^0$ or
the first Morava $K$-theory with respect to the prime $p$.
\end{cor}
\begin{proof}
Let $Y=\SB(B)$ denote the Severi--Brauer variety of $B$. We need to check the conditions of Proposition~\ref{moravaaffine2}.
The variety $Y$ is a geometrically cellular $\nu_1$-variety (see \cite[Section~7.2]{Me03}).
Thus, by Section~\ref{eulerc} its $A^*$-Euler characteristic is invertible.

Over a point $y\in Y$ the variety $\MS(B,c)$ is isomorphic to $\SL_p$, since $\MS(B,c)$ over $F(y)$
is an $\SL_p$-torsor over $F(y)$ and $H^1_{\et}(F(y),\SL_p)$ is trivial. Since $\GL_p$ is an open subvariety in
$\mathbb{A}^{p^2}$, by the localization sequence $\Omega^*(\GL_p)=\laz$. Moreover, $\GL_p$ is isomorphic
as a variety (not as a group scheme) to $\SL_p\times\mathbb{G}_m$ with the isomorphism sending a matrix $\alpha$
to the pair $(\beta,\det\alpha)$ where $\beta$ is obtained from $\alpha$ by dividing its first row by $\det\alpha$.
The composite morphism $$\SL_p\hookrightarrow\GL_p\xrightarrow{\simeq}\SL_p\times \mathbb{G}_m\to\SL_p,$$ where the first morphism
is the natural embedding and the last morphism is the projection, is the identity. Taking pullbacks in this
sequence, one gets that $\Omega^*(\SL_p)=\laz$ and, hence,
$A^*(\SL_p)=A^*(F(y))$ for $A^*$ as in the statement of the present corollary. We are done.
\end{proof}

Let $J$ be an Albert algebra over $F$ (see \cite[Chapter~IX]{Inv}) and $N_J$ denote the cubic norm form on $J$.
For $d\in F^\times$ consider the variety $$Z=\{\alpha\in J\mid N_J(\alpha)=d\}.$$ The group $G$ of isometries
of $N_J$ is a group of type $^1{}\E_6$ and it acts on $Z$ geometrically transitively. Note also that $Z$ is in general anisotropic.

\begin{cor}
In the above notation the natural map 
$K^0(F)\to K^0(Z)$ is an isomorphism.
\end{cor}
\begin{proof}
Let $Y$ be the variety of Borel subgroups of the group $G$. We need to check the conditions of
Proposition~\ref{moravaaffine2}. The $K^0$-Euler characteristic of $Y$ is invertible, since $Y$ is
geometrically cellular.

Let $y\in Y$ be a point. Then $G$ splits over $F(y)$, the variety $Z$ has a rational point over $F(y)$
and its stabilizer is the split group of type $\FF_4$, i.e., $Z$ is isomorphic to $\E_6/\FF_4$ over $F(y)$,
where $\E_6$ and $\FF_4$ stand for the split groups of the respective Dynkin types.
Finally, by \cite[Theorem~2]{Y16} $K^0(\E_6/\FF_4)\simeq K^0(F(y))$.
We are done.
\end{proof}
\begin{rem}
In \cite{Y16} Yakerson computes the whole higher $K$-theory of
twisted forms of $\E_6/\FF_4$ by means of cocycles from $Z^1(F,\FF_4)$. Note that such twisted forms are isotropic.
\end{rem}

Consider the Witt-ring of the field $F$ and denote by $I$ its fundamental ideal.

\begin{prop}\label{moravaproj}
Let $m\ge 2$ and set $p=2$.
A non-degenerate even-dimensional quadratic form $q$ belongs to $I^m$ iff the $K(n)^*$-motive of the respective projective quadric is split for all $0\le n<m-1$.
\end{prop}
\begin{rem}
Note that by Proposition~\ref{prop:height_tate_motive} below the $K(m-2)^*$-motive of the respective quadric splits iff its $K(n)^*$-motive is split for all $0\le n<m-1$. We do not use this in the proof of Proposition~\ref{moravaproj}.
\end{rem}
\begin{proof}[Proof of Proposition~\ref{moravaproj}]
Note that the statement of the proposition is clear for $m=2$, since $K(0)^*$ is defined as $\CH\otimes\qq$ and $I^2$ consists of all non-degenerate even-dimensional quadratic forms with trivial discriminant. Therefore, $q$ belongs to $I^2$ iff the $K(0)^*$-motive of the respective projective quadric is split. So, we assume that $m>2$. It is also clear that it suffices to proof the proposition for $0<n<m-1$.

Assume that $q$ does not belong to $I^m$. Let $1\le s<m$ be the maximal integer with $q\in I^s$, and assume that the $K(s-1)^*$-motive of the respective quadric is split. If $s=1$, then as was mentioned above its $K(0)^*$-motive is not split.
So, we can assume that $s>1$.

By \cite[Theorem~2.10]{OVV07} there exists a field extension $K$ of $F$ such that the anisotropic part of $q_K$ is similar to an anisotropic $s$-fold Pfister form, say $q'$.  Thus, $q_K$ is isomorphic to an orthogonal sum
of $q'$ (up to a scalar multiple) and a hyperbolic form. Let $Q$ and $Q'$ be the projective quadrics over $K$ associated
with $q_K$ and $q'$ resp. By \cite[Proposition~2]{Rost}
the Chow motive of $Q$ is isomorphic to a sum of Tate motives and a Tate twist of the motive of $Q'$.
Therefore, by Vishik--Yagita \cite[Section~2]{ViYa07} the same decomposition holds for the cobordism motives and, hence, for the Morava motives.
But by Proposition~\ref{pro62} the $K(s-1)^*$-motive of $Q'$ and, hence, of $Q$ is not split.
Contradiction.

Conversely, assume that $q$ belongs to $I^m$ and let $Q$ be the respective projective quadric. Let $1\le n<m-1$.
Then we can present $q$ as a finite sum of (up to proportionality)
$m$-fold Pfister forms. We prove our statement using induction on the length of such a presentation in the Witt-ring.               
If the length is zero, i.e., if $q$ is a split form, then the $K(n)^*$-motive of $Q$ is split for all $n$.

Let $\alpha$ be an $m$-fold Pfister form in the decomposition of $q$. Let $X^{\aff}$ be the affine norm quadric of
dimension $2^{n+1}-1$ corresponding to a subsymbol of $\alpha$ from $H_{\et}^{n+2}(F,\zz/2)$ (note that $n+2\le m$).
Then the length of $q$ over $F(X^{\aff})$ is strictly smaller than the length of $q$ over $F$.

Applying Lemma~\ref{lem610} to the varieties $X=Q\times Q$ and $Z=X^{\aff}$  and using Corollary~\ref{moravaaffine} we obtain that the pullback of the natural projection $$K(n)^*(Q\times Q)\to K(n)^*(X^{\aff}\times Q\times Q)$$ is surjective.

But by the localization sequence $$K(n)^*(X^{\aff}\times Q\times Q)\to K(n)^*((Q\times Q)_{F(X^{\aff})})$$
is surjective. By the induction hypothesis on the length of $q$, the restriction homomorphism
$$K(n)^*((Q\times Q)_{F(X^{\aff})})\to K(n)^*((Q\times Q)_{\widetilde F})$$
to a splitting field $\widetilde F$ of $Q_{F(X^{\aff})}$ is surjective. Therefore, the restriction homomorphism $$K(n)^*(Q\times Q)\to K(n)^*((Q\times Q)_{\widetilde F})$$
is surjective. In particular, since the projectors
for the Morava motive of $Q$ lie in $K(n)^*(Q\times Q)$, it follows from Rost nilpotence that the $K(n)^*$-motive of $Q$ over $F$ is split.
\end{proof}

\begin{rem}
The same statement with a similar proof holds for the variety of totally isotropic
subspaces of dimension $k$ for all $1\le k\le 
(\dim q)/2$. Note that in the case of a Pfister 
form, the motive of such a variety is still a 
direct sum of Rost motives.
\end{rem}

The same proof also shows the following proposition.

\begin{prop}\label{moravaodd}
If $q\in I^m$ for some $m\ge 2$ and $q'=q\perp \langle c\rangle$ for some $c\in F^\times$, then the $K(n)^*$-motive of the
respective projective quadric $q'=0$ is split for all $0\le n<m-1$.

Conversely, if $q'$ is an odd-dimensional quadratic form such that the $K(n)^*$-motive of the respective projective quadric $q'=0$ is split for some $n\ge 0$, then in the Witt ring $q'=q\perp\langle c\rangle$ for some $q\in I^{n+2}$ and $c\in F^\times$.
\end{prop}
\begin{proof}
The proof is almost verbatim as of Proposition~\ref{moravaproj}. 
For the reader's convenience we sketch the proof of the second part of the proposition.

Set $q=q'\perp\langle -\mathrm{disc}(q')\rangle$ in the Witt ring. Then $q\in I^{n+2}$. Indeed, otherwise as in the proof of Proposition~\ref{moravaproj} over some field extension $K$ of $F$ the anisotropic part of the form $q_K$ will be similar to an anisotropic $s$-fold Pfister form with $2\le s<n+2$.
Then the motive of the respective Pfister quadric is a direct sum of Tate twists of Rost motives $R_{s}$. But by Proposition~\ref{pro62} the $K(n)^*$-motive $R_s$ is not split.
Besides, it follows from \cite[Theorem~17]{Rost} that the motive of the quadric $q'=0$ contains over $K$ a Tate twist of $R_{s}$.
Hence, the $K(n)^*$-motive of the projective quadric $q'=0$ is not split over $K$ and, hence, is not split over $F$. 
\end{proof}

\section{$K(n)$-split varieties and $p$-torsion in Chow groups}\label{sec:morava_general}

In this section we obtain several general results concerning Morava $K$-theories.
First, 
with the help of the Landweber--Novikov 
operations we prove that if a projective homogeneous variety
is split with respect to $K(n)^*$, then it is split 
with respect to $K(m)^*$ for $m\colon 1\le m \le n$ 
(Corollary~\ref{zzzzz}). Recall that by results of \cite{ViYa07} if the motive of a smooth projective variety $X$ is split with respect to the Chow theory, then it is split for every oriented theory.

Thus, to prove the non-splitting of the $p$-local Chow motive of a 
projective homogeneous variety one
could consecutively check the splitting of the
Morava motives $M_{K(1)}$, $M_{K(2)}$, etc. If one of these motives is non-split,
then the $p$-local Chow motive is non-split as well. Conversely,
if all Morava motives are split, then the $p$-local Chow motive is split as well. In fact, by Corollary~\ref{cor:morava_to_chow_motive} below
it suffices to consider the $n$-th Morava $K$-theory such that $p^n$ is greater than or equal to the dimension of the variety. In this sense one could interpret $\CH^*\otimes\zz_{(p)}$ as a Morava $K$-theory of an infinite height.

Secondly, we investigate properties of smooth projective geometrically cellular varieties $X$
for which the pullback restriction map $K(n)^*(X)\rarr K(n)^*(\overline X)$
is an isomorphism. 
Using symmetric operations we show in Theorem~\ref{th:no_torsion}
that $\CH^i(X)$ has no $p$-torsion for such varieties
where $i \le \frac{p^n-1}{p-1}$.
Finally, in Theorem~\ref{th:general_gamma_app}
we use the gamma filtration on $K(n)^*$
to prove finiteness
of $p$-torsion in Chow groups of such varieties
in codimensions up to $p^n$.

\begin{ntt}[Landweber--Novikov operations and split $K(n)^*$-motives]\label{sec:LandNov}

As was mentioned in Section~\ref{freetheory} every formal group law $(R,\mathcal{F}_R)$
yields a free theory $\Omega^*\otimes_{\LL}R$.
It is natural to ask about relationships between free theories corresponding to isomorphic formal group laws.
For simplicity, taking
an isomorphism between formal group laws $(R,\mathcal{F}_R)$ and $(R, \mathcal{F}'_R)$
which is the identity on $R$ 
(i.e., it is a change of the ``parameter'' of the formal group law), one obtains an isomorphism
of the presheaves of rings of corresponding free theories. Moreover, a considerable part of such isomorphisms
can be obtained via the specialization 
of the total Landweber--Novikov operations 
on the level of algebraic cobordism. These operations put severe constraints on the structure of
the algebraic cobordism as an $\LL$-module, and we will use
this in the study of the Morava $K$-theories of $K(n)^*$-split varieties.

Recall that there exists a graded Hopf algebroid $(\LL,\LL B)$ which
represents the formal group laws and strict isomorphisms 
between them (see e.g. \cite[Appendix~1.1.1, Appendix~2.1.16]{Rav}),
 where $\LL B = \LL[b_1,b_2,\ldots]$.
The total Landweber--Novikov operation 
$$S^{tot}_{L-N}\colon \Omega^*\rarr \Omega^*\ot_\LL \LL B$$ is a multiplicative operation
which in some sense corresponds to the universal strict isomorphism
of formal group laws, see \cite[Example 3.9]{Vish1}.

\begin{prop}[{\cite[Proposition~2.10]{SechCob}}]
The action of the Landweber--Novikov operations makes $\Omega^*$ into
a functor to the graded comodules over the Hopf algebroid $(\LL,\LL B)$. 
\end{prop}

In particular, $\mathbb{L}=\Omega^*(\Spec F)$ is canonically a 
comodule over $(\LL,\LL B)$, and its subcomodules
are the same as the ideals 
which are invariant with respect to the Landweber--Novikov operations.
The only non-zero prime ideals among them are $I(p,m)=(p,v_1,\ldots,v_{m-1})$ 
and $I(p)=\cup_{m} I(p,m)$,
 where $p$ is a prime number
and $v_i$'s
are $\nu_i$-elements
(\cite[Theorem~2.2]{Land}).

The situation with $BP^*$ is very similar with $\Omega^*$. 
For every smooth variety $X$ the $BP$-module $BP^*(X)$
is a direct summand of $\Omega^*(X)\ot\Zp$ (see e.g. \cite[Proposition~2.4]{SechCob}),
and one can restrict the action of the Landweber--Novikov operations
to $BP^*(X)$ which makes it a graded comodule over 
the Hopf algebroid $(BP, BP_*BP)$ (for the latter see \cite[Appendix~2.1.27]{Rav}).
In particular,
there is an action of the Landweber--Novikov operations on $BP$
and the only non-zero invariant prime ideals 
are of the form $I(k)=(p,v_1,\ldots,v_{k-1})$ (\cite[Theorem~$2.2_{BP}$]{Land}).

The abelian category of comodules over $(\LL,\LL B)$ (or over $(BP, BP_*BP)$)
was extensively studied
by topologists.
Note also that $(\LL,\LL B)$ is canonically isomorphic to $(MU_*, MU_*(MU))$,
and the latter notation is often used in the literature.

\begin{prop}[{\cite[Theorem~3.3]{Land2}}, {\cite[Theorem~2.2, 2.3, $2.2_{BP}, 2.3_{BP}$]{Land}}]
\label{prop:landweber_filtration}

Let $M$ be a graded comodule over $(\LL,\LL B)$ (over $(BP, BP_*BP)$, respectively)
 which is finitely presented as an $\LL$-module (as a $BP$-module, respectively).
Then $M$ has a filtration $$M=M_0\supset M_1\supset \ldots \supset M_d=0$$
such that for every $i$ the module $M_i/M_{i+1}$ is isomorphic to $\LL/I(p_i,n_i)$ 
or $\LL$ ($BP/I(m_i)$ or $BP$, respectively) after a shift of grading,
where $p_i$ is a prime number and $n_i$ is a positive integer
($m_i$ is a positive integer, respectively).
\end{prop}

\begin{cor}\label{cor:support_coh_mfg}
Fix a prime $p$ and for $s\ge 1$ denote by $K(s)\cong \ZZ_{(p)}[v_s,v_s^{-1}]$ the {$\LL$-algebra}
corresponding to a choice of a formal group law for a Morava $K$-theory $K(s)^*$.
If for $M$ as in Proposition~\ref{prop:landweber_filtration} we have $M\ot_\LL K(n)=0$ for some $n\ge 1$, 
then $M\ot_\LL K(m)=0$ for all $m\colon 1\le m<n$.
\end{cor}
\begin{proof}
We call by a filtration of an $(\LL,\LL B)$-comodule $M$ just any filtration
from Proposition~\ref{prop:landweber_filtration}. 
We will prove  a stronger statement by induction on the minimal length $d$ of a filtration of $M$.
Namely, if $M\ot_\LL K(n)=0$, then $\Tor_i^{\LL}(M, K(m))=0$ for all $i\ge 0$ and $m\le n$,
and the graded factors of the filtration on $M$ can be only of the form $\LL/I(q,k)$
with $q\neq p$ or with $q=p$ and $n\le k-1$.

For the base of induction $d=1$ we just need to check  the statement for modules $\LL/I(q,k)$.
In both cases (if $q\neq p$ or if $q=p$ and $n\le k-1$) 
$\Tor_i^{\LL}(\LL/I(q,k), K(m))=0$ because it is naturally
both a $K(m)$-module and an $\LL/I(q,k)$-module (compatible with the structure of an $\LL$-module).
If $q\ne p$, then $q$ is invertible in $K(m)=\zz_{(p)}[v_m,v_m^{-1}]$.
If $q=p$ and $m\le k-1$, then $v_m$ is invertible in $K(m)$ and is zero in $\LL/I(q,k)$. 
Therefore, $\Tor_i^{\LL}(\LL/I(q,k), K(m))=0$ for all $i\ge 0$. 
Clearly, if $q=p$ and $n>k-1$, then $\LL/I(q,k) \ot_\LL K(n) \neq 0$.

For the induction step suppose that $M$ has a filtration of length $d+1$,
which means that there exists a short exact sequence of $(\LL,\LL B)$-comodules:
$$0\rarr N\rarr M \rarr \LL/I(q,k)\rarr 0,$$ where $N$ has a filtration of length $d$.
Tensoring this sequence with $K(n)$, we see from the above that either $q\neq p$ or $q=p$ and $n\le k-1$.
Tensoring with $K(m)$, ${1\le m \le n}$, we obtain that $N\ot K(n)=0$
and $\Tor_i^{\LL}(N, K(m))\simeq\Tor_i^{\LL}(M, K(m))$ for all $i\ge 0$ and ${1\le m\le n}$.
We then apply the induction hypothesis to $N$ to conclude that $\Tor_i^{\LL}(M, K(m))=0$ for all $i\ge 0$ and all $m$, $1\le m\le n$.
\end{proof}

\begin{rem}
Let $\CH_{(p)}$ denote the coefficient ring of $\CH^*\otimes\zz_{(p)}$
and let $M$ be as in Corollary~\ref{cor:support_coh_mfg}.
Analogously one can show that if $M\otimes_{\laz}\CH_{(p)}=0$, then $M\otimes_{\laz}K(m)=0$ for all $m\ge 1$.
\end{rem}

\begin{rem}
The language of stacks might provide a more geometric view on the statement above.
Indeed, the category of comodules over the Hopf algebroid $(\LL,\LL B)$ can be identified
with the category of quasi-coherent sheaves over the stack of formal groups 
$\mathcal{M}_{fg}$ (see e.g. \cite{Naum}).
Working modulo $p$ this stack has an exhaustive descending filtration by closed substacks
where the $n$-th piece of it $\mathcal{M}_{fg}^{\ge n}$ 
classifies formal groups of height bigger than or equal to $n$.
Moreover, these substacks are the only irreducible closed (reduced) substacks,
and $\mathcal{M}_{fg}^{\ge n+1}$ is in some sense a divisor in $\mathcal{M}_{fg}^{\ge n}$ 
whose complement has a unique geometric point which corresponds 
to the $n$-th Morava $K$-theory.

The support of a coherent sheaf $\mathcal{G}$ over $\mathcal{M}_{fg}$ 
is closed, and therefore the reduced support is the closed substack  
$\mathcal{M}_{fg}^{\ge m}$ for some $m$. 
In particular, the fibre of $\mathcal{G}$ over the points 
corresponding to the $n$-th Morava $K$-theory is zero if $n<m$
and non-zero if $m\ge n$. This gives a vague explanation of Corollary~\ref{cor:support_coh_mfg}.
\end{rem}

\begin{cor}\label{cor:Kn-trivial-ideals-of-BP}
Let $C$ be a finitely presented $BP$-module
endowed with the structure of a $(BP, BP_*BP)$-comodule.

If $C\ot_{BP} K(n)=0$, then $C\ot_{BP} BP[v_n^{-1}]=0$.
\end{cor}
\begin{proof}
By Proposition~\ref{prop:landweber_filtration}
the $BP$-module $C$ has a filtration with the graded factors $BP/I(k_i)$.
The same proof as of Corollary~\ref{cor:support_coh_mfg} 
in which one replaces $\LL$ with $BP$ and $I(p,k)$ with $I(k)$
shows that if $C \otimes_{BP} K(n)=0$,
then for the graded factors of the filtration above 
 for all $i$ we have $n\le k_i-1$, i.e. $v_n\in I(k_i)$. 
The claim follows.
\end{proof}

The following lemma is straight-forward.

\begin{lem}\label{lm:RNP_surj}
Let $X$ be a geometrically cellular smooth projective variety over a field $F$ 
and let $A^*$ be a free oriented cohomology theory.
Assume that the $A^*$-motive $M_A(X)$ satisfies the Rost nilpotence property. 
Denote $\overline X=X\times_{F}\overline F$.
Then the following statements are equivalent:
\begin{enumerate}
\item $M_A(X)$ is split;
\item the restriction map $A^*(X\times X)\rarr A^*(\overline{X}\times_{\overline F} \overline{X})$ is an isomorphism;
\item the restriction map $A^*(X\times X)\rarr A^*(\overline{X}\times_{\overline F} \overline{X})$ is a surjection;
\item the restriction map $A^*(X)\rarr A^*(\overline{X})$ is an isomorphism;
\item the restriction map $A^*(X)\rarr A^*(\overline{X})$ is a surjection.
\end{enumerate}
\end{lem}
\begin{proof}
To prove the implication $(5)\Rightarrow(3)$ note that  $\overline X$ is cellular,
its motive is split and all elements in $A^*(\overline X)$ 
and, therefore, in $A^*(\overline{X}\times_{\overline F} \overline{X})$ are rational.
The implication $(3)\Rightarrow (1)$ follows from Rost nilpotence.
\end{proof}

\begin{cor}
Assume that two free theories $A^*$ and  $B^*$ are isomorphic as presheaves of sets, and
the motives $M_A(X)$ and $M_B(X)$ satisfy the Rost nilpotence property.
Then $M_A(X)$ is split iff $M_B(X)$ is split.
\end{cor}
\begin{proof}
Indeed, an isomorphism between $A^*$ and $B^*$ commutes with the change of the base field. Thus, whenever one of the maps $A^*(X)\rarr A^*(\overline{X})$, $B^*(X)\rarr B^*(\overline{X})$
is surjective, so is the other one.
\end{proof}

In particular, it follows from the above corollary 
and \cite[Corollary~4.5]{GV18}
that for a fixed prime $p$, a fixed integer $n$
and a projective homogeneous variety $X$
there is a well-defined property
for $M_{K(n)}(X)$ to be split which does not depend on the choice of an 
$n$-th Morava $K$-theory.

\begin{prop}\label{prop:height_tate_motive}
Let $1\le m\le n$, and
let $X$ be a smooth projective geometrically cellular variety such that $M_{K(m)}(X)$ 
satisfies the Rost nilpotency property.

If $M_{K(n)}(X)$ is split, then $M_{K(m)}(X)$ is split.
\end{prop}
\begin{proof}
By Lemma \ref{lm:RNP_surj} it is sufficient to prove that 
the map $K(m)^*(X)\rarr K(m)^*(\overline{X})$ is surjective, 
whenever $K(n)^*(X)\rarr K(n)^*(\overline{X})$ is so.

Consider the following short exact sequence of $(\LL,\LL B)$-comodules:
$$ \Omega^*(X)\xrarr{\rho} \Omega^*(\overline{X})\rarr C \rarr 0.$$

Clearly, the map $\rho\ot_\LL K(m)$ is surjective iff $C\ot_\LL K(m)=0$. However, ${C\ot_\LL K(n)=0}$ by the assumption, and $C$ is a coherent
$\LL$-module by \cite[Proposition~2.21, Remark~2.24]{SechCob}.
 Therefore, Corollary~\ref{cor:support_coh_mfg} applies.
\end{proof}
\begin{cor}\label{zzzzz}
If $X$ is a projective homogeneous variety
such that $M_{K(n)}(X)$ is split, 
then $M_{K(m)}(X)$ is split for all $1\le m\le n$.
\end{cor}
\begin{proof}
By \cite[Corollary~4.5]{GV18} for every free theory $A^*$ 
the motive $M_A(X)$ satisfies the Rost nilpotence property.
\end{proof}

\begin{cor}\label{cor:morava_to_chow_motive}
Let $X$ be a projective homogeneous variety with $\dim X\le p^n$.
Then the Chow motive of $X$ with $\zz_{(p)}$-coefficients is split if and only if the $K(n)^*$-motive of $X$ is split.
\end{cor}
\begin{proof}
If the Chow motive of $X$ with $\zz_{(p)}$-coefficients is split, then obviously the $K(n)^*$-motive of $X$ is split as well.

Assume now that the $K(n)^*$-motive of $X$ is split.
The operations $$c_i^{\CH}\colon K(n)^*(X)\to\CH^i(X)\otimes\zz_{(p)}$$ are surjective for $i\le p^n$ and commute with extensions of scalars (see Theorem~\ref{th:morava_gamma_properties}).

Therefore, condition~(5) of Lemma~\ref{lm:RNP_surj} is satisfied for $A^*=\CH^*\otimes\zz_{(p)}$. This
implies the corollary.
\end{proof}
\end{ntt}

\begin{ntt}[Symmetric operations of Vishik and $K(n)^*$-split varieties]\label{sec:topfilt_symm}
We have used above the Landweber--Novikov operations 
which are stable (\cite[Definition~3.4]{Vish1})
and provide constraints on the structure of cobordism which do not ``see'' the grading.
Being interested in the Chow groups and in the topological filtration of small codimension we employ more subtle unstable operations,
among which the most powerful are symmetric operations.

Recall that Vishik has defined symmetric operations in algebraic cobordism
first for $p=2$ in \cite{Vish_Symm_2} using elaborate and elegant constructions
 and then for all primes in \cite{Vish_Symm_all} 
 using \cite[Theorems~5.1]{Vish1, Vish2} 
 classifying all operations.
We follow the latter approach and explain several properties of these operations.

Fix a set of integers $\bar{i}=\{i_j\mid 0<j<p\}$ of 
all representatives of non-zero integers modulo $p$, 
and denote $\mathbf{i}=\prod_{j=1}^{p-1} i_j$. 
There exists a Quillen-type Steenrod operation in algebraic cobordism
$$ St(\bar{i})\colon\Omega^*\rarr \Omega^*[\mathbf{i}^{-1}][[t]][t^{-1}],$$
which induces a morphism of formal group laws
uniquely defined by the 
power series $\gamma(x)=x\prod_{j=1}^{p-1}(x+_\Omega i_j\cdot_\Omega t)$.
We will sometimes drop $\bar{i}$ from
the notation of $St$.

\begin{thm}[Vishik, {\cite[Theorem~7.1]{Vish_Symm_all}}]

{\ }

There exists a unique operation 
$\Phi(\bar{i})\colon\Omega^*\rarr \Omega^*[\mathbf{i}^{-1}][t^{-1}]$,
called {\it the symmetric operation},
such that 
\begin{equation}\label{eq:symm}
(\square^p-St(\bar{i})-\frac{p\cdot_\Omega t}{t}\Phi(\bar{i}))\colon\Omega^*
\rarr \Omega^*[\mathbf{i}^{-1}][[t]]t,
\end{equation}
where $\square^p$ is the $p$-power operation.
\end{thm}

It is convenient to use ``slices'' of the symmetric operation $\Phi(\bar{i})$
defined as the coefficients of the monomials $t^{l}$ for $l\le 0$.
We will denote these operations as $\Phi_l(\bar{i})=\Phi_l$.

Fix a prime $p$ and for simplicity we will work $p$-locally,
in particular, using $BP^*$ instead of $\Omega^*$.
Recall that there exists a multiplicative projector
on $\Omega^*\ot\Zp$ making $BP^*$ a direct summand of $\Omega^*\ot\Zp$.
This allows one to restrict the symmetric operation to $BP^*$
even though it is non-additive (see \cite[Section 3]{Vish_Lazard}).

Recall that following Hazewinkel we have chosen the generators $v_n$ of
the ring $BP$ (see Section~\ref{sec:BP}).
Symmetric operations $\Phi_l$
allow to ``divide'' certain elements of $BP^*$ by elements $v_n$
as was observed e.g. in \cite[Section~3.2]{SechCob}. 
The following is an instance of this property. 

\begin{prop}\label{prop:symm_op}
Let $k>0$ and let $\alpha \in BP^{-k(p^n-1)}$ such that $\alpha \equiv v_n^k \mod I(n)$.

Then $\Phi_{-k(p-1)(p^n-1)-(p^n-1)}(\alpha) \equiv -v_n^{k-1} \mod I(n)$.
\end{prop}
\begin{proof}
By the definition of the symmetric operation
we have the following identity 
in the ring $BP[[t]][t^{-1}]$
in the coefficients of $t^{\le 0}$:

\begin{equation}\label{eq:symm_alpha}
 \alpha^p-St(\alpha) =^{t^{\le 0}} [p]\cdot \Phi(\alpha), \qquad [p]:=\frac{p\cdot_{BP} t}{t}.
\end{equation}

Recall that $St$ is a generalized specialization of the total Landweber--Novikov operation (\cite[p.~977]{Vish_Lazard}),
i.e. it can be obtained 
from the unstable total Landweber--Novikov operation $\Omega^* \rarr \Omega^*\ot_\LL \LL[b_0^{\pm 1}, b_1, b_2, \ldots]$
defined by the inverse Todd genus $\sum_{i=0}^\infty b_i t^{i}$.
Therefore, it is an (infinite) $BP$-linear combination of the Landweber--Novikov operations 
 (see \cite[Section~3]{Vish_Lazard} for more details). In particular, $St$ preserves the ideals $I(n)$, hence
 $St(\alpha)\equiv St(v_n^k) \mod I(n)$.

It follows from the Riemann--Roch theorem for multiplicative operations
 (see e.g. \cite[Lemma~2.16]{SechCob}) that $$St(v_n)\equiv v_n t^{-(p-1)(p^n-1)} \mod I(n).$$

The series $[p]$ appearing above is graded of degree 0 where we take $\deg t=1$.
Moreover, it starts with $p$, and therefore modulo $I(n)$ the smallest power of $t$
appearing in it is equal to $p^n-1$ and its coefficient is proportional to $v_n$.
The choice of Hazewinkel implies that it is exactly $v_n$ (\cite[A2.2.4]{Rav}),
i.e. $[p] = v_n t^{p^n-1} +\mathrm{higher\ degree\ terms}$.

Combining all this together, equation~\eqref{eq:symm_alpha} modulo $I(n)$
looks as:
$$ v_n^{kp}-v_n^kt^{-k(p-1)(p^n-1)}\equiv^{t^{\le 0}} (v_nt^{p^n-1}+\mathrm{higher\ degree\ terms})\Phi(\alpha) \mod I(n),$$
from which the statement follows
using \cite[Lemma~3.3]{SechCob} and the fact that $BP/I(n)$ is an integral domain.
\end{proof}

The previous proposition can be used to study rational elements in the $BP^*$-theory
as the following lemma shows. It will be a crucial step in the proof of Theorem~\ref{th:no_torsion} below.

For an element $z\in BP^r(X)$ we write $\deg z=r$.

\begin{lem}\label{lm:symm_top_shift}
Let $f\colon X\rarr Y$ be a morphism of smooth quasi-projective varieties.
Let $z\in BP^r(X)$. Assume that $r>\frac{p^n-1}{p-1}$ and for some $k\ge 0$ the element $v_n^k z$ belongs to the image of $BP^*(Y)$ 
under the map $f^*$.

Then there exists a homogeneous element $\beta\in BP$ such that the element
$\beta z$  belongs to the image of $f^*$ modulo $\tau^{r+1} BP^*(X)$ 
and
\begin{itemize}
    \item $\beta \equiv v_n^b \mod I(n)$ for some $b\ge 0$;
    \item $\deg(\beta z)=r-b(p^n-1)>\frac{p^n-1}{p-1}$.
\end{itemize}
\end{lem}
\begin{proof}
Let $x\in BP^*(Y)$ be such that $f^*(x) = v_n^k z$. 
We will apply the symmetric operation $\Phi$ to $x$
several times
producing the needed element $y\in BP^*(Y)$ such that $f^*(y)\equiv \beta z \mod \tau^{r+1} BP^*(X)$ for $\beta$ as in the statement of the proposition.
Since all operations commute with pullbacks, we just have to calculate how the operation $\Phi$ 
acts on $v_n^k z$.

Moreover, all operations preserve the topological filtration, 
and by \cite[Proposition~7.14]{Vish_Symm_all} 
there is a simple description of the action of the symmetric operation on $\mathrm{gr}^\bullet_\tau BP^*$.
Namely, for any $\lambda \in BP$ and $z$ as above we have
 $$\Phi(\lambda z) \equiv \mathbf{i}^r\cdot t^{r(p-1)}\cdot \Phi_{\le -r(p-1)}(\lambda) z \mod \tau^{r+1}BP^*(X),$$
 where $\Phi_{\le -r(p-1)}(\lambda)$ is the part of the polynomial $\Phi(\lambda)\in BP[t^{-1}]$
 with the degree of $t$ no greater than $-r(p-1)$.

Thus, to be able to use Proposition~\ref{prop:symm_op} and ``divide'' $v_n^k z$ 
by $v_n$ we need that $k>0$ (if $k=0$ we do not have to do anything) 
and $$-k(p-1)(p^n-1)-(p^n-1)\le -r(p-1).$$
Equivalently, $(r-k(p^n-1))(p-1)\le p^n-1$ or $\deg (v_n^k z)\le \frac{p^n-1}{p-1}$.
We can continue this process 
until we get the desired element $\beta z$ modulo $\tau^{r+1} BP^*(X)$
where ${\deg(\beta z)=r-b(p^n-1)>\frac{p^n-1}{p-1}}$.
\end{proof}

\begin{thm}\label{th:no_torsion}
Let $X$ be a smooth projective geometrically cellular variety 
such that the pullback map $f^*\colon K(n)^*(X)\rarr K(n)^*(\overline X)$ 
is an isomorphism, where $\overline X = X\times_{F} \overline{F}$.

Then the pullback maps 
\begin{equation}\label{eq:gr_Kn_CH_isom}
\gr^r_\tau K(n)^*(X)\rarr \gr^r_\tau K(n)^*(\overline X), \quad
\CH^r(X)\ot\Zp\rarr \CH^r(\overline X)\ot\Zp
\end{equation}
are isomorphisms for $r \le \frac{p^n-1}{p-1}$.

In particular, $\CH^r(X)$ has no $p$-torsion for all $r \le \frac{p^n-1}{p-1}$.
\end{thm}
\begin{proof}
For a smooth projective cellular variety $Y$ and a free theory $A^*$
the $A$-module $A^*(Y)$ is free, generated by chosen classes 
of desingularizations of (closed) cells. We will call these elements classes of cells,
and the codimension of the class of a cell is the codimension of the corresponding cell.
Moreover, the $r$-th part of the topological filtration on $A^*(Y)$ 
is generated by the cells of codimension no less than $r$.

It follows that $\CH^r(\overline X)\ot\Zp$ is torsion-free,
and the last claim of the theorem follows from claim~\eqref{eq:gr_Kn_CH_isom}.

For simplicity of notation we switch now from Morava $K$-theories with $\Z{p}[v_n, v_n^{-1}]$-coefficients
to Morava $K$-theories with $\Z{p}$-coefficients by sending $v_n$ to $1$. Clearly, this does not affect neither assumptions,
nor conclusions of the theorem. 

Under the assumptions of the theorem 
the pullback map from $\gr^r_\tau K(n)^*(X)$ to  
$\gr_{\tau}^r K(n)^*(\overline{X})$ is surjective for $r\colon 0\le r \le p^n-1$
since 
\begin{equation}\label{eq:gr_kn}
\gr^r_\tau \tilde{K}(n)^*= \tilde{K}(n)^r/\tau^{r+p^n-1}\tilde{K}(n)^r    
\end{equation}
in this range of $r$ 
by Proposition~\ref{prop:morava_grading}(3).
On the other hand, $\gr^r_\tau K(n)^*(\overline X)$ is a free $\Zp$-module
generated by the classes of cells of codimension $r$. 
Thus, to prove the theorem it is sufficient to show that
preimages under $f^*$ of classes of all cells of codimension 
greater than $\frac{p^n-1}{p-1}$ lie in $\tau^{>\frac{p^n-1}{p-1}} K(n)^*(X)$. 
Indeed, this would imply that $f^*$ is an isomorphism between 
$\tau^{r+p^n-1}K(n)^r(X)$ and $\tau^{r+p^n-1}K(n)^r(\overline X)$ for $r\le \frac{p^n-1}{p-1}$,
and therefore $f^*$ is also an isomorphism on $\gr^r_\tau \tilde{K}(n)^r$ by formula~\eqref{eq:gr_kn}.

For the class $z$ of a cell in $BP^r(\overline X)$
denote by $z_{K(n)}$ its image in $K(n)^*(\overline X)$.
Also abusing notation we denote the preimage of this element in $K(n)^*(X)$
under $f^*$ by the same letter.

We now argue by decreasing induction on $r$ from $\dim X+1$ to $\frac{p^n-1}{p-1}+1$
that $${z_{K(n)}\in \tau^{\frac{p^n-1}{p-1}+1} K(n)^*(X)}.$$

{\bf Base of induction} is trivial for $r=\dim X+1$, since $BP^r(X)=0$.

{\bf Induction step.}
Assume that for all classes $z_s$ of cells in $BP^{>r}(\overline X)$
the classes $z_{s,K(n)}$ lie in $\tau^{\frac{p^n-1}{p-1}+1} K(n)^*(X)$.

Denote by $C$ the cokernel 
of the map $BP^*(X)\rarr BP^*(\overline X)$.
It is a finitely presented $BP$-module 
with the structure of a comodule over the Hopf algebroid $(BP, BP_*BP)$ 
(\cite[Proposition~2.21, Remark~2.24]{SechCob}).
Moreover, $C\ot K(n)=0$ by the assumptions of the theorem,
and, therefore, by Corollary~\ref{cor:Kn-trivial-ideals-of-BP}
 the pullback map $$BP^*(X)[v_n^{-1}]\rarr BP^*(\overline X)[v_n^{-1}]$$
is surjective.
In particular, for every class of a cell $z\in BP^r(\overline X)$ of codimension $r$
there exists $k \ge 0$ such that $v_n^{k} z$ is a rational element.

If $z\in BP^r(\overline X)$ is the class of a cell of codimension $r>\frac{p^n-1}{p-1}$,
then by Lemma~\ref{lm:symm_top_shift} applied to $f:\overline X \rarr X$ we obtain that
the element $\beta z +\sum_s \alpha_s z_s \in BP^j(\overline X)$ is rational
for some $j>\frac{p^n-1}{p-1}$, $\alpha_s,\beta\in BP$ such that $\beta$ maps to $1$
in $K(n)$ and $z_s$ are classes of cells of bigger codimension (recall that $\tau^{r+1} BP^*(\overline X)$ is generated
by cells of codimension at least $r+1$). 

Let $y$ be an element of $BP^j(X)$ which maps to $\beta z +\sum_s \alpha_s z_s\in BP^j(\overline X)$ under the pullback map.
Then $y\in \tau^j BP^j(X)$, since $BP^j=\tau^j BP^j$ (the last formula holds by the definition of the topological filtration and by the fact that $BP$ contains no elements of strictly positive degree).
Therefore, the image of $y$ in $K(n)^*(X)$ also lies in $\tau^j K(n)^*(X)$,
and at the same time its image in $K(n)^*(\overline X)$ has the form  $z_{K(n)}+\sum [\alpha_s]_{K(n)} z_{s, K(n)}$
where $[\alpha_s]_{K(n)}$ is the image of $\alpha_s$ under the canonical morphism $BP\rarr K(n)=\Zp$.
However, by the induction assumption the preimages under the isomorphism $f^*$ of the elements
$z_{s, K(n)}$ already lie in $\tau^{\frac{p^n-1}{p-1}+1}K(n)^*(X)$,
hence the claim.

\medskip

As explained above it follows that the pullback map 
$\gr^i_\tau K(n)^*(X)\rarr \gr^i_\tau K(n)^*(\overline X)$ is an isomorphism 
for $i\le \frac{p^n-1}{p-1}$. The operation $c_i^{\CH}\colon\gr^i_\tau K(n)^*\rarr \CH^i\ot\Zp$
commutes with pullbacks by definition and induces an isomorphism
for $i\le p^n$  by Theorem~\ref{th:morava_gamma_properties},~\ref{item:chern_isom_top}).
It follows that the map $\CH^i(X)\ot\Zp\rarr \CH^i(\overline X)\ot\Zp$ 
is also an isomorphism for $i\le \frac{p^n-1}{p-1}$.
\end{proof}
\end{ntt}

\begin{ntt}[Finiteness of torsion in Chow groups via the gamma filtration]
We consider the Morava $K$-theory $K(n)^*$ with $v_n$ set to be $1$.

Above we have used the topological filtration on Morava $K$-theories
to show that there is no $p$-torsion in Chow groups  of certain varieties
up to codimension $\frac{p^n-1}{p-1}$. 
However, calculating graded factors of the topological filtration $\gr^i_\tau K(n)^*$
in the range between $\frac{p^n-1}{p-1}+1$ and $p^n$ is
a very complicated task even though it would still 
yield $\CH^i\ot\Zp$ by Theorem~\ref{th:morava_gamma_properties}.
Yet another approach to estimate $p$-torsion in Chow groups
is to use the gamma filtration instead of the topological filtration.

\begin{thm}\label{th:general_gamma_app}
Fix a prime $p$ and let $K(n)^*$ be the corresponding $n$-th
Morava $K$-theory. Assume that $X$ is a geometrically cellular smooth projective variety
such that the restriction map $K(n)^*(X)\rarr K(n)^*(\overline{X})$
is an isomorphism, where $\overline{X}=X\times_F \overline{F}$.

Then the $p$-torsion in $\CH^j(X)$
is a quotient of the $p$-torsion in $\gr^j_\gamma K(n)^*(\overline{X})$ for $j\le p^n$.

In particular, the $p$-torsion in the Chow groups of $X$ 
is finite in codimensions up to $p^n$
and it can be bounded based on the variety $\overline{X}$ only.
\end{thm}
\begin{proof}
As the gamma filtration is defined using the operations
which commute with pullbacks by definition, 
the gamma filtrations
on $K(n)^*(X)$ and $K(n)^*(\overline{X})$ 
coincide via the change of the base field.
Therefore, the graded pieces of the gamma filtration of $X$ depend only on $\overline{X}$.

Note that as $\overline{X}$ is cellular, its Chow motive is of Tate type, i.e. is split, 
and, therefore, its algebraic cobordism motive
is of Tate type as well. Therefore, $K(n)^*(\overline{X})$ is a finitely generated free $\Zp$-module
generated by the classes of desingularizations of the closed cells. 
This proves that the graded pieces of the gamma filtration (on both $K(n)^*(X)$ 
and $K(n)^*(\overline{X})$) are finitely generated,
and thus have finite torsion. 

By Theorem~\ref{th:morava_gamma_properties},~\ref{item:chern_surj}) and~\ref{item:gr_chern_isom_rat})
we have surjective
additive maps $$c_j^{\CH}\colon\gr_\gamma^j K(n)^*(X)\rarr \CH^j(X)\otimes \Zp$$
for $j\le p^n$, which are isomorphisms rationally.
Therefore, $\CH^j(X)$ has finite torsion for every $j\le p^n$ which is bounded above by the torsion of
$\gr_\gamma^j K(n)^*(X)$.
\end{proof}

An advantage of this approach is that the calculations are of a purely combinatorial nature
and are often amenable as we will show in the case of quadrics in the next section.
However, the bounds obtained by the gamma filtration are not exact in general
(see Remark~\ref{rem:top_vs_gamma}).
\end{ntt}

\section{Bounds on torsion in Chow groups via Morava $K$-theory}\label{bounds}

In this section we will provide some bounds on torsion in Chow groups of quadrics. Before doing this we would like
to summarize known results in this direction. We apologize in advance in case we forgot to mention some contributions.

\begin{ntt}[Karpenko's bounds in small codimensions]

\begin{prop}[Karpenko]\label{th:karpenko_bounds}
Let $Q$ be a smooth projective anisotropic quadric of dimension $D$ defined over a field of characterstic not 2.
\begin{itemize}
\item \cite[Theorem~6.1]{Karp90}: $\mathrm{Tors} \CH^2(Q)=0$ for $D>6$;
\item \cite[Theorem~6.1]{Karp95}: $\mathrm{Tors} \CH^3(Q)=0$ for $D>10$;
\item \cite[Theorem~8.5]{Karp95}: $\mathrm{Tors} \CH^4(Q)=0$ for $D>22$;
\end{itemize}
\end{prop}

When dimension of a quadric is smaller than in the above proposition,
Karpenko gives some bounds for the torsion in $\CH^3$ and $\CH^4$
and explicitly computes $\CH^2$ (see \cite[Theorem~6.1]{Karp90}, \cite{Karp96}).

We remark at this point that there are examples of quadrics having infinite torsion in $\CH^4$ (see \cite{KM91}).

\end{ntt}

\begin{ntt}[Rost motives and excellent quadrics]

The Chow groups of Pfister quadrics and more generally of excellent quadrics are explicitely known.
This was computed by Rost in \cite[Theorem~5]{Ro90}, see also \cite[Theorem~7.1, Theorem~8.1]{KM02}. More generally,
Yagita computed the multiplicative structure of the Chow rings of excellent quadrics (see \cite{Ya08}).

In particular, the following result holds.
\begin{prop}[Rost]
Let $Q_{\alpha}$ be the Pfister quadric
corresponding to a pure non-zero symbol $\alpha\in H^n_{\et}(F,\zz/2)$, $n\ge 3$.

Then $\mathrm{Tors}\CH^i(Q_\alpha)=0$ for $i<2^{n-2}$
and $\mathrm{Tors}\CH^{2^{n-2}}(Q_\alpha)=\ZZ/2$.
\end{prop}
\end{ntt}

\begin{ntt}[Vishik's calculation for generalized Albert's forms]

Consider a generalized Albert form of dimension $6\cdot 2^r$ over $F$, i.e. a form of the type
$\rho\otimes\varphi$, 
where $\rho$
is an Albert form, i.e.
$\rho = \langle a,b,-ab,-c,-d,cd\rangle$ for some $a,b,c,d\in F^\times$, and $\varphi$ is an $r$-fold Pfister form.
Note that by \cite[Lemma~1.4]{Vish-Albert} there exist anisotropic generalized Albert forms of dimension $6\cdot 2^r$ over suitable fields.

For a quadratic form $q$ denote by $Q$ 
the respective projective quadric.

\begin{prop}[Vishik, {\cite[Main Theorem]{Vish-Albert}}]
If a generalized Albert forms $q$ of dimension
$6\cdot 2^r$ with $r\ge 1$ is anisotropic, then $\mathrm{Tors} \CH^{2^r+1}(Q)\neq 0$.
\end{prop}

Below we will show that there is no torsion in $\CH^{j}(Q)$ for $j<2^r+1$ (Corollary~\ref{cr:bounds_genalbert}).
\end{ntt}

Finally, there are numerous results with computations of the Chow groups of generic quadrics and generic orthogonal Grassmannians
(see \cite{Ka17}, \cite{P16}, \cite{SV14}).

\begin{ntt}[The gamma and the topological filtration on Morava $K$-theories of quadrics]

In this section $p=2$.
Denote by $\overline{Q}$ a split quadric of dimension $D$ 
and assume that $D\ge 2^{n+2}-3$.

Denote by $d:=[D/2]$ 
the dimension of the maximal isotropic projective space inside $\overline Q$
and by $\iota\colon\mathbb{P}^d\rarr \overline{Q}$ the corresponding inclusion map.
Denote by $h\in K(n)^*(\overline Q)$ the first Chern class of the canonical line bundle $\mathcal{O}(1)$.
Abusing notation we will denote by the same letter the pullbacks of this class
along restrictions to open subsets of $\overline Q$.

The following proposition is well-known. We consider the Morava $K$-theory $K(n)^*$ with $v_n$ set to be $1$.

\begin{prop}\label{prop88}
The natural linear projection map $b\colon\overline{Q}\setminus \mathbb{P}^d \rarr \mathbb{P}^d$
induces an isomorphism
 $b^*\colon K(n)^*(\mathbb{P}^d)\rarr K(n)^*(\overline{Q}\setminus \mathbb{P}^d)$.
 
Moreover, there is a short (split) exact sequence of abelian groups:
$$0\rarr \oplus_{s=0}^d \Z{2}l_s \xrarr{\iota_*} K(n)^*(\overline{Q}) \xrarr{\pi^*}
K(n)^*(\mathbb{P}^d)\rarr 0 ,$$
where the map $\pi^*$ is a morphism of rings compatible with the gamma filtration and $l_s$ 
is the class of a linear projective space inside $\overline Q$ of dimension $s$.
\end{prop}
\begin{proof}
Since the statement of the proposition is well-known, we only
sketch the proof.

Let $(V,q)$ be the quadratic space of dimension $D+2$ with a split quadratic form $q$.
Let $W\subset V$ be the maximal totally isotropic subspace of $V$. Then $\dim W=d+1$.
The map $\iota_*$ is the push-forward of the embedding $\iota\colon\mathbb{P}^d=\mathbb{P}(W)\hookrightarrow\overline{Q}$.

The quadratic form $q$ induces a natural linear map $V\to W^*$.
This map induces a morphism $b\colon \overline{Q}\setminus\mathbb{P}(W)\hookrightarrow \mathbb{P}(V)\setminus\mathbb{P}(W)\to\mathbb{P}(W^*)=\mathbb{P}^d $ which is an affine bundle of rank $D-d$. Therefore, by homotopy invariance the homomorphism $b^*$ is an isomorphism.

Let $\theta\colon\overline Q\setminus\mathbb{P}(W)\hookrightarrow\overline Q$ be the open embedding. Then by the localization axiom the sequence
$$K(n)^*(\mathbb{P}^d)\xrightarrow{\iota_*}K(n)^*(\overline{Q})\xrightarrow{\theta^*}K(n)^*(\overline{Q}\setminus\mathbb{P}^d)\to 0$$ is exact. Now the homomorphism $\pi^*$ is defined as $ (b^*)^{-1}\circ\theta^*$. Using the fact that all objects here are free $\zz_{(2)}$-modules of suitable ranks one can check that the resulting exact sequence is exact on the left and is split.
\end{proof}

Note that $\pi^*$ in Proposition~\ref{prop88} induces surjective maps of abelian groups 
$${\gr_\gamma^rK(n)^*(\overline{Q})\rarr \gr_\gamma^r K(n)^*(\mathbb{P}^d)}$$ for $r\ge 0$.
A direct calculation shows that $\gr_\gamma^r K(n)^*(\mathbb{P}^d)$ has no torsion
for all $r$, i.e. it equals $\Z{2}$ for $0\le r\le d$ and $0$ for $r>d$.
Thus, we have $${h^r \in \gamma^rK(n)^*(\overline Q) \setminus \gamma^{r+1}K(n)^*(\overline Q)}$$
for $0\le r\le d$
(one could also see this using 
rational comparisons of Theorem~\ref{th:morava_gamma_properties}).

We claim that the torsion in $\gr^r_\gamma K(n)^*(\overline{Q})$ is generated
by elements of $\mathrm{Im}\, \iota_*$. Indeed, take an element $\alpha$ from $\gamma^r K(n)^*(\overline{Q})$. By Proposition~\ref{prop88} one can express $\alpha$ as a linear combination of elements from $\mathrm{Im}\, \iota_*$ and elements $h^k$ with $k\ge r$. Taking $\alpha$ modulo $\gamma^{r+1} K(n)^*(\overline{Q})$ we may assume that it is a linear combination of elements from $\mathrm{Im}\, \iota_*$ and the element $h^r$, say, with a coefficient $a$.

If $\alpha$ gives a torsion element in $\gr^r_\gamma K(n)^*(\overline{Q})$, then it maps to a torsion element in $\gr^r_\gamma K(n)^*(\mathbb{P}^d)$, hence to $0$. But it maps to $a\hslash^r$ where $\hslash$ is the first Chern class of $\mathcal{O}(1)$ on $\mathbb{P}^d$. Therefore, $a=0$.

We recall the multiplication structure in $K(n)^*(\overline Q)$.

\begin{prop}\label{prop:mult_morava_quadric}
\begin{enumerate}
\item\label{item:transversal_intersection}
We have $ h\cdot l_i = l_{i-1}$ where we denote $l_{-1}=0$.

\item\label{item:mult_middle_quadric}
If the dimension of the quadric is odd,
then $h^{d+1} \equiv 2l_d \mod (l_j\mid 0\le j<d)$.
Moreover, $h^{d+1}$ is expressible in terms of $l_j$ 
with $j\equiv d \mod 2^n-1$. 
\end{enumerate}
\end{prop}
\begin{proof}
To prove (\ref{item:transversal_intersection}) note that $h$ can be represented 
by a general hyperplane section of $\overline Q$, so that it intersects transversally the linear subspace
representing the class $l_i$. The product $h\cdot l_i$ is represented by their intersection, 
which is then a linear subspace of dimension one less.

Part (\ref{item:mult_middle_quadric}) follows from the well-known multiplication in the Chow ring of $\overline Q$.
\end{proof}

For simplicity of notation set $l_r=0$ for $r<0$.

Let $D-d\equiv 1+j \mod 2^n-1$ where $j\in [0,2^n-2]$.

From now on we consider a non-split smooth projective quadric $Q$ of positive dimension such that the restriction map
\begin{align}\label{mmmm}
K(n)^*(Q)\to K(n)^*(\overline Q)
\end{align}
to a splitting field of $Q$ is an isomorphism.  Note that in this case $\dim Q\ge 2^{n+2}-3$. Indeed, in view of Lemma~\ref{lm:RNP_surj}, if $\dim Q$ is even, this follows from Proposition~\ref{moravaproj} and the Arason--Pfister Hauptsatz. If $\dim Q$ is odd, then by Proposition~\ref{moravaodd} the respective quadratic form is of the type $f\perp\langle c\rangle$ for some anisotropic form $f\in I^{n+2}$ and $c\in F^\times$. Therefore, since the dimension of the anisotropic part of 
$f\perp\langle c\rangle$ is at least $\dim f-1$, it follows from the Arason--Pfister Hauptsatz that $\dim Q\ge 2^{n+2}-3$.

Abusing notation we will consider the elements $h, l_i$ of $K(n)^*(\overline Q)$ defined above
also as the corresponding elements of $K(n)^*(Q)$ with respect to the isomorphism~\eqref{mmmm}.

\begin{lem}\label{lm:mult_shift_gamma_quadric}
Let $k\in[0,d]$.
Assume that the element $l_k$ lies in $${\gamma^rK(n)^*(Q) \mod \oplus_{s<k} \Z{2}l_s}$$  
(resp. in $\tau^rK(n)^*(Q) \mod \oplus_{s<k} \Z{2}l_s$) for some $r\ge 1$.

Then for every $u\ge 0$ the element
 $l_{k-u}$ lies in $\gamma^{r+u}K(n)^*(Q)$ 
 (resp. in $\tau^{r+u}K(n)^*(Q)$).
\end{lem}
\begin{proof}
The proof is the same for the gamma and for the topological filtration
 and exploits only its multiplicativity.
We confine ourselves to the case of the gamma filtration. 
By our assumptions we have $l_k+\sum_{s<k} a_sl_{s} \in \gamma^rK(n)^*(Q)$
for some $a_s\in\Z{2}$.

We prove the statement by decreasing induction starting with the highest $u=k+1$. In this case $l_{-1}=0$
and the claim is trivial. 

By Proposition~\ref{prop:mult_morava_quadric} we have 
$$h^u \cdot (l_k+\sum_{s<k} a_sl_s) = l_{k-u} + \sum_{u\le s<k} a_s l_{s-u}.$$

The left-hand side lies in $\gamma^{r+u}K(n)^*(Q)$ by the multiplicativity of the gamma filtration
and the fact that $h\in \gamma^1K(n)^*(Q)$,
while the ``tail'' of the right-hand side lies in $\gamma^{r+u+1}K(n)^*(Q)$ by the induction assumption.
Therefore, we have ${l_{k-u}\in \gamma^{r+u}K(n)^*(Q)}$.
\end{proof}

Denote by $H$ the first Chern class of the canonical line bundle $\mathcal{O}(1)$ on $Q$
in the Brown--Peterson cohomology. Again abusing notation, denote by the same letter
the corresponding class in $BP^*(\overline{Q})$.
Denote by $L_r \in BP^*(\overline Q)$ the class of a linear subspace inside $Q$ of dimension $r$.
Note that the canonical map of theories $$\pi_{K(n)}\colon BP^*(\overline Q)\rarr K(n)^*(\overline Q)$$ 
sends $H$ to $h$ and $L_r$ to $l_r$.

\begin{lem}\label{lm:ld_top_morava}
$l_d \in \tau^{j+2^n} K(n)^*(Q)$.
\end{lem}
\begin{proof}
One could argue as in the proof of Theorem~\ref{th:no_torsion}
to show that $l_d \in \tau^{2^n} K(n)^*(Q)$.

A more direct approach of the use of Theorem~\ref{th:no_torsion}
is the following. Let $i$ be the maximal positive integer
such that $l_d\in \tau^i K(n)^*(Q)$. If $i<2^n$, then $l_d$ defines 
a non-trivial element of the group $\gr^i_\tau K(n)^*(Q)$.
However, this group maps isomorphically to $\gr^i_\tau K(n)^*(\overline Q)$
where the class of $l_d$ is zero. Contradiction and, therefore, $i\ge 2^n$,
i.e. $l_d\in \tau^{2^n} K(n)^*(Q)$.

However, $l_d\in K(n)^{1+j}(Q)$ and $\tau^{2^n} K(n)^{1+j}(Q)=\tau^{j+2^n}K(n)^{1+j}(Q)$
by Proposition~\ref{prop:morava_grading}(3).
This implies the claim.
\end{proof}

\begin{prop}
In the notation of this section we have
\begin{enumerate}
\item  $\gr^s_\tau K(n)^*(Q)=\Z{2}$ for $1\le s\le 2^n-1$;
\item if $j\neq 0$, then  $\gr^{2^n}_\tau K(n)^*(Q)=\Z{2}$;
\item if $j=0$, and the dimension of the quadric is odd,
 then the torsion subgroup of $\gr^{2^n}_\gamma K(n)^*(Q)$
 is at most $\ZZ/2$;
\item if $j=0$ and the dimension of the quadric is even,
 let $d=1+r(2^n-1)$ for some\footnote{The case $r=2$ is the well-known case of a Pfister quadric.} $r\ge 2$.
 If $r$ is even, then the torsion in $\gr^{2^n}_\gamma K(n)^*(Q)$ is at most $\ZZ/2$.
 If $r$ is odd, then the torsion in $\gr^{2^n}_\gamma K(n)^*(Q)$ is at most $\ZZ/2^s$,
 where ${s=\min (\nu_2(r-1)+2, 2^n)}$. Here we denote by $\nu_2$ the $2$-adic valuation.
\end{enumerate}
\end{prop}
\begin{proof}
(1) This follows from Theorem~\ref{th:no_torsion}.

(2) If $j\neq 0$, then by Lemma~\ref{lm:ld_top_morava} 
the element $l_d$ lies in $\tau^{2^n+1}K(n)^*(Q)$
and therefore, by Lemma~\ref{lm:mult_shift_gamma_quadric} the same holds for $l_s$, $s<d$. Thus, 
the graded factors $\gr^s_\tau K(n)^*(Q)$ for $s\le 2^n$ 
have to be generated by some power of $h$ and have no torsion. 

(3, 4) If $j=0$, then we will show now that $l_d \in \gamma^{2^n}K(n)^*(Q)\subset \tau^{2^n}K(n)^*(Q)$. 
Let $\iota\colon\mathbb{P}^{d}\hookrightarrow \overline Q$ be the inclusion 
of the maximal isotropic linear subspace.
In order to calculate $c^{K(n)}_{2^n}(l_d)=c^{K(n)}_{2^n}(\iota_* 1_{\mathbb{P}^{d}})$
 we apply the generalized Riemann--Roch formula (Corollary~\ref{cr:riemann-roch}).
Using Proposition~\ref{prop:constant_cpn}
we have $c^{K(n)}_{2^n}(\iota_* 1_{\mathbb{P}^{d}})$ is equal to $${e_r l_{d}+\sum_{s>0} b_sl_{d-s(2^n-1)}}$$ with $b_s\in\Z{2}$,
where $r$ is such that $d=1+r(2^n-1)$ and $e_r \in \Z{2}^\times$. This element
lies in $\gamma^{2^n}K(n)^*(Q)$ by the definition of the gamma filtration.

By Lemma~\ref{lm:mult_shift_gamma_quadric} we obtain that all other elements $l_s$, $s<d$,
lie in the higher parts of the gamma filtration. 
It follows that the torsion in the group $\gr^{2^n}_\gamma K(n)^*(Q)$ 
is generated by $l_d$.

(only 3) If the dimension of the quadric is odd, then 
by Proposition~\ref{prop:mult_morava_quadric} 
we have $h^{d+1} = 2l_d + \sum_{s>0} \beta_s l_{d-s(2^n-1)}$
for some $\beta_s \in \Z{2}$. By the multiplicativity of the gamma filtration
this element lies in $\gamma^{d+1} K(n)^*(Q)$. Recall that $d\ge 2^{n+1}-2$ by our assumptions, and therefore, 
$d+1> 2^n$.
 Thus, by the results above ${l_{d-s(2^n-1)}\in \gamma^{2^{n}+1} K(n)^*(Q)}$ for $s>0$,
 and we obtain that $2l_d \in \gamma^{2^{n}+1} K(n)^*(Q)$.
 This proves the claim.

(only 4) Let us consider the element $\chi(l_d) \in \gamma^{2^{n+1}-1}K(n)^1(Q)$
for the operation $\chi$ from Proposition~\ref{prop:constant_chern_classes2}.
Using the Riemann--Roch formula and Proposition~\ref{prop:constant_chern_classes2}, \ref{item:coef_2pn-1})
we obtain that this element is equal to $g_r l_d+\sum_{s>0} b_s l_{d-s(2^n-1)}$
for some $b_s \in \Z{2}$ ($g_r$ was defined in Proposition~\ref{prop:constant_chern_classes2}). Since the elements $l_{d-s(2^n-1)}$ lie in $\gamma^{2^{n+1}-1}K(n)^*(Q)$,
we obtain that $g_r l_d \in \gamma^{2^{n+1}-1}K(n)^*(Q)$.
If $r$ is even, then $g_r \in 2\Z{2}^\times$.
If $r$ is odd, then $\nu_2(g_r)=\nu_2(r-1)+2$.

If we use the operation $\psi$ from Proposition~\ref{prop:constant_chern_classes2}
instead of $\chi$, we obtain that $${2^{2^n} l_d \in \gamma^{2^{n+1}-1}K(n)^*(Q)}.$$
The result now follows.
\end{proof}

Combining this together with
Theorem~\ref{th:morava_gamma_properties}
and Propositions~\ref{moravaproj} and \ref{moravaodd} we obtain the following theorem.

\begin{thm}\label{th:mult_torsion_quadric}
Let $Q$ be a smooth quadric  of positive dimension over a field $F$ such that corresponding quadratic form $q$
 lies either in the ideal $I^{n+2}$
or in the set $\langle c\rangle+I^{n+2}$ inside the Witt ring for some $c\in F^\times$.
Let $D$ be the dimension of $Q$, $d:=[D/2]$,
and let $j\in[0,2^n-2]$ be such that $D-d\equiv 1+j\mod 2^n-1$.

Then $\CH^{0\le * \le 2^n-1}(Q)=\ZZ$ and 
\begin{enumerate}
\item if $j\neq 0$, then $\CH^{2^n}(Q)=\ZZ$.

\item if $j=0$,  and the dimension of the quadric is odd,
 then the torsion in $\CH^{2^n}(Q)$ is at most $\ZZ/2$;
 
\item if $j=0$ and the dimension of the quadric is even,
$d=1+r(2^n-1)$,
then the torsion in $\CH^{2^n}(Q)$ is at most $\ZZ/2^s$
where $s=1$, if $r$ is even, and $${s=\min (\nu_2(r-1)+2, 2^n)}$$ otherwise.
 Here we denote by $\nu_2$ the $2$-adic valuation.
\end{enumerate}
\end{thm}

\begin{rem}\label{rem:top_vs_gamma}
One can show that the estimates one gets using just the gamma filtration are not so strong if $j\neq 0$.
Namely, if $j\neq 0$ one obtains $\ZZ/2$ in the components $\CH^{\ge j+1}$.
This shows that the graded factors of the gamma filtration do not give exact bounds for 
the topological filtration even in small codimensions.
\end{rem}

\begin{cor}\label{cr:bounds_genalbert}
Let $q$ be a generalized anisotropic Albert form of dimension $6\cdot 2^r$.
Then $\mathrm{Tors} \CH^j(Q)=0$ for all $j<2^r+1$.
\end{cor}
\begin{proof}
Indeed, the Albert form lies in $I^2(F)$, and therefore, $q$ lies in $I^{r+2}(F)$.
Since $d\equiv 2 \mod (2^r-1)$, Theorem~\ref{th:mult_torsion_quadric} implies the claim.
\end{proof}

Vishik has communicated to the authors
that one can show the above corollary
using techniques of \cite{SV14}.

\end{ntt}

\section{Morava $K$-theory and cohomological invariants}\label{sec:morava_coh_inv}

In this section we relate the Morava $K$-theory with cohomological invariants of algebraic group; see also Proposition~\ref{moravaproj}.

\begin{thm}\label{moravarost}
Let $p$ be a prime number.
Let $G$ be a simple algebraic group over $F$ and let $X$ be the variety of Borel subgroups of $G$. Then
\begin{enumerate}
\item $G$ is of inner type iff the $K(0)^*$-motive of $X$ is split.
\item (Panin). Assume that $G$ is of inner type. All Tits algebras of $G$ are split iff the $K^0$-motive with integral coefficients
of $X$ is split.
\item Assume that $G$ is of inner type and the $p$-components of the Tits algebras 
of $G$ are split. Then the $p$-component of the Rost invariant of $G$ is zero iff the $K(2)^*$-motive of $X$ is split.
\item Let $p=2$. Assume that $G$ is of type $\E_8$.
Then $G$ is split by an odd degree field
extension iff the $K(m)^*$-motive of $X$ is split for some $m\ge 4$ iff the $K(m)^*$-motive of $X$ is split for all $m\ge 4$.
\end{enumerate}
\end{thm}
\begin{proof}
(1) Recall that $K(0)^*=\CH^*\otimes\qq$ by definition.
If $G$ is of inner type, then it is well-known that the Chow motive of $X$ with rational coefficients is split (e.g.
this follows from \cite[Theorems~2.2 and 4.2]{Pa94}, since $K^0$ and $\CH^*$ are isomorphic theories with rational coefficients).
On the other hand, if $G$ is of outer type, then the absolute Galois group of $F$ acts non-trivially on the Chow group of $X_{\Fsep}$
(see \cite[Section~2.1]{MT95} for the description of the action on the Picard group of $X_{\Fsep}$). Therefore, the Chow motive of $X$ with
any coefficients cannot be split in this case.

(2) Follows from \cite{Pa94}; see also Section~\ref{titsk0}.

(3) First we make several standard reductions.
Since all prime numbers coprime to $p$ are invertible in the coefficient ring of the Morava $K$-theory,
by transfer argument we are free to take finite field extensions of the base field of degree coprime to $p$.
Hence we can assume that not only the $p$-components of the Tits algebras are split, but that the Tits algebras are completely split (and the same for the Rost invariant).

{\bf Types $\A$ and $\C$.}
If $G$ is a group of inner type $\A$ or $\C$ with trivial Tits algebras, then $G$ is split and the statement
follows. Indeed, by \cite[\S26]{Inv} the group $G$ is isogenous to $\SL_1(A)$ for a central simple algebra $A$ or, respectively,
to $\Sp(B,\sigma)$ for a central simple algebra $B$ with a symplectic involution $\sigma$. 
By \cite[\S27.B]{Inv} the algebra $A$, respectively, the algebra $B$ is a Tits algebra of $G$. 
Therefore, if $A$, respectively, $B$ is split, then $G$ is split, and the statement of the proposition is obvious.

{\bf Types $\B$ and $\D$.}
If $G$ is a group of inner type $\B$ or $\D$, then $G$ is isogenous to $\Spin(V,q)$
or, respectively, to $\Spin(D,\tau)$ for an odd-dimensional quadratic space $(V,q)$ 
or, respectively, for an algebra $D$ with an orthogonal involution $\tau$ with trivial discriminant.
By \cite[\S27.B]{Inv} the even Clifford algebra $C_0(V,q)$, respectively, the algebra $D$ and the Clifford algebras $C^{\pm}(D,\tau)$ 
are Tits algebras of $G$. Therefore, if the Tits algebras of $G$ are split, we are in the situation of quadratic forms.

Now the statement of the proposition follows from Proposition~\ref{moravaproj} (in the even-dimensional case) and from Proposition~\ref{moravaodd} (in the odd-dimensional case). Indeed, an even-dimensional quadratic form with trivial discriminant lies in $I^4$ (resp. an odd-dimensional quadratic form lies in $I^4+\langle c\rangle$ for some $c\in F^\times$) iff its Clifford and its Rost invariants are zero (see \cite[\S31.B]{Inv} for a description of the Rost invariant in the case of quadratic forms).

{\bf Exceptional types.}
Let now $G$ be a group of an exceptional type. Taking coprime to $p$ field extensions we assume that our base field is $p$-special.
Assume that the $K(2)^*$-motive of $X$ is split, but the Rost invariant of $G$ is not trivial. 

There is a field extension $K$ of $F$ such that the Rost invariant of $G_K$
is a non-zero pure symbol.
Indeed, for groups of types $\E_6$, $\FF_4$, $\G_2$, $\E_7$ with $p=3$ and $\E_8$ with $p=5$ this is already the case for $K=F$ (see \cite[Part~II]{Ga09}).

If $G$ is of type $\E_7$ with $p=2$, then by \cite[Theorem~5.7]{PS10} the variety $Y$ of maximal parabolic subgroups of $G$ of type $6$ (enumeration of simple roots follows Bourbaki) is not generically split. Over its function field $K=F(Y)$ the anisotropic kernel of $G_K$ is of type $\D_4$ and, thus, the Rost invariant of $G_K$ is a non-zero pure symbol.

If $G$ is of type $\E_8$ with $p=2$, then by \cite[Theorem~5.7]{PS10} one
can take $K=F(Y)$, where $Y$ is the variety of maximal parabolic subgroups of $G$ of type $6$ (the anisotropic kernel of $G_K$ will be again of type $\D_4$), and if $G$ is of type $\E_8$ with $p=3$, then one
can take $K=F(Y)$, where $Y$ is the variety of maximal parabolic subgroups of $G$ of type $7$ (the anisotropic kernel of $G_K$ will be of type $\E_6$).

In all cases the motive of $X_K$ is a direct sum of Rost motives corresponding to this non-zero symbol
of degree $3$ (see \cite{PSZ08}). This gives a contradiction with Proposition~\ref{pro62}.

Conversely, if the Rost invariant of $G$
is zero and $G$ is not of type $\E_8$ with $p=2$, then by \cite[Theorem~0.5]{Ga01} (for exceptional groups different from $\E_8$), \cite{C94} and \cite[Proposition~15.5]{Ga09} (for $\E_8$ at the prime $5$), \cite{C10} and \cite[Section~10c]{GPS16} (for $\E_8$ at the prime $3$) the group $G$ is split and the statement of the proposition follows.

Therefore, it remains to consider the case when $G$ is a group of type $\E_8$ with trivial Rost invariant. By \cite[Theorem~8.7]{Sem13} $G$ has an invariant $u\in H^5_{\et}(F,\zz/2)$
such that for every field extension $K/F$ the invariant $u_K=0$ iff $G_K$ splits over a field extension
of $K$ of odd degree.
Exactly as in the proof of Proposition~\ref{moravaproj} (note that we can represent $u$ by a quadratic form from $I^5$) we can pass to a splitting field $\widetilde F$ of $u$ such that the restriction homomorphism $K(2)^*(X\times X)\to K(2)^*((X\times X)_{\widetilde F})$ is surjective. Therefore, by Rost nilpotence the $K(2)^*$-motive of $X$ is split.

(4) If $G$ is split by an odd degree field extension, then the $K(m)^*$-motives of $X$ are split for all $m$, since $p=2$.
Conversely, if $G$ does not split over an odd degree field extension of $F$ and
the even component of the Rost invariant of $G$ is non-trivial, then by item (3) the $K(2)^*$-motive of $X$ is not split
and, hence, by Proposition~\ref{prop:height_tate_motive} the $K(m)^*$-motives are not split for all $m\ge 2$.

Besides, if $G$ does not split over an odd degree field extension of $F$ and the even component of the Rost invariant
of $G$ is trivial, then the invariant $u$ is defined and is non-zero.
By \cite[Theorem~2.10]{OVV07} there is field extension $K$ of $F$ such that $u_K$ is a non-zero pure symbol.
Over $K$ the motive of $X$ is a direct sum of Rost motives corresponding to $u_K$. By Proposition~\ref{pro62} the $K(m)^*$-Rost
motives for a symbol of degree $5$ are not split, if $m\ge 4$.
\end{proof}

\vspace{-0.2cm}

Finally we remark that sequence~\eqref{voev} can be used to define the Rost invariant in general,
the invariant $f_5$ for groups of type $\FF_4$ (see \cite[\S40]{Inv}) and an invariant of degree $5$ for groups of type $\E_8$ with trivial Rost invariant (see \cite{Sem13}).
Namely, for the Rost invariant let $G$ be a simple simply-connected algebraic group over $F$.
Let $Y$ be a $G$-torsor and set $n=3$. Then sequence~\eqref{voev} gives an exact sequence
$$0\to H_{\mathcal{M}}^{3,2}(\mathcal{X}_Y,\qq/\zz)\to \Ker\big(H^3_{\et}(F,\qq/\zz(2))\to H_{\et}^3(F(Y),\qq/\zz(2))\big)\to 0$$
But by sequence~\eqref{rostseq} $\Ker\big(H^3_{\et}(F,\qq/\zz(2))\to H_{\et}^3(F(Y),\qq/\zz(2))\big)$ is a finite cyclic
group. Therefore, $H_{\mathcal{M}}^{3,2}(\mathcal{X}_Y,\qq/\zz)$ is a finite cyclic group and the Rost invariant
of $Y$ is the image of $1\in H_{\mathcal{M}}^{3,2}(\mathcal{X}_Y,\qq/\zz)$ in $H_{\et}^3(F,\qq/\zz(2))$.

To construct invariants of degree $5$ for $\FF_4$ (resp. for $\E_8$) one takes $n=5$ and $Y$ to be the variety of parabolic subgroups
of type $4$ for $\FF_4$ (the enumeration of simple roots follows Bourbaki) and resp. the variety of parabolic subgroups of any type for $\E_8$.
In both cases $H^{5,4}_{\mathcal{M}}(\mathcal{X}_Y,\qq/\zz)$ is cyclic of order $2$ and the invariant is the
image of the only non-zero element of $H^{5,4}_{\mathcal{M}}(\mathcal{X}_Y,\qq/\zz)$ in $H^5_{\et}(F,\qq/\zz(4))$; see \cite{Sem13}.

%%%%%%%%%%%%%%% REFERENCES %%%%%%%%%%%%%%%

\medskip

\medskip

\noindent
\sc{Pavel Sechin\\
Mathematisches Institut, Ruprecht-Karls-Universit\"at Heidelberg,\\ 
Mathematikon, Im Neuenheimer Feld 205, 69120 Heidelberg, Germany;\\
National Research University Higher School of Economics, Moscow, Russia\\
{\tt psechin@mathi.uni-heidelberg.de}
}

\medskip

\medskip

\noindent
\sc{Nikita Semenov\\
Mathematisches Institut, Ludwig-Maximilians-Universit\"at M\"unchen, Theresienstr. 39,
D-80333 M\"unchen, Germany}\\
{\tt semenov@math.lmu.de}

\end{document}